\documentclass[hyperref]{article}

\usepackage{tikz}
\usepackage{graphicx} 
\usepackage{float} 
\usepackage{amssymb}
\usepackage{amsmath}
\usepackage{mathtools}

\usepackage[thmmarks,amsmath]{ntheorem} 
\usepackage{bm} 
\usepackage[colorlinks,linkcolor=cyan,citecolor=cyan]{hyperref}
\usepackage{extarrows}
\usepackage[hang,flushmargin]{footmisc} 
\usepackage[square,comma,sort&compress,numbers]{natbib} 
\usepackage{mathrsfs} 
\usepackage[font=footnotesize,skip=0pt,textfont=rm,labelfont=rm]{caption,subcaption} 
\usepackage{booktabs} 
\usepackage{tocloft}
\usepackage{stmaryrd}

\usepackage{verbatim}
{
    \theoremstyle{nonumberplain}
    \theoremheaderfont{\bfseries}
    \theorembodyfont{\normalfont}
    \theoremsymbol{\mbox{$\Box$}}
    \newtheorem{proof}{Proof}
}

\usepackage{theorem}
\newtheorem{theorem}{Theorem}[section]
\newtheorem{lemma}{Lemma}[section]

{
    \theoremheaderfont{\bfseries}
    \theorembodyfont{\normalfont}
    
}
\graphicspath{{figure/}}                                 
\usepackage[a4paper]{geometry}                           
\begin{document}
\newcommand{\topcaption}{%
\setlength{\abovecaptionskip}{0.cm}%
\setlength{\belowcaptionskip}{0.cm}%
\caption}

\title{\bf A discontinuous Galerkin based multiscale method for heterogeneous elastic wave equations} 
\date{}
\author{\sffamily Zhongqian Wang$^1$, Shubin Fu$^{2,*}$, Zishang Li$^1$, Eric Chung$^1$\\
    {\sffamily\small $^1$ Department of Mathematics, The Chinese University of Hong Kong, Hong Kong SAR }\\
    {\sffamily\small $^2$ Department of Mathematics, University of Wisconsin - Madison, USA }}
\renewcommand{\thefootnote}{\fnsymbol{footnote}}
\footnotetext[1]{Corresponding author: shubinfu89@gmail.com. }
\maketitle

{\noindent\small{\bf Abstract:}
In this paper, we develop a local multiscale model reduction strategy for the elastic wave equation in strongly heterogeneous media, which is achieved by solving the problem in a coarse mesh with multiscale basis functions. We use the  interior penalty discontinuous Galerkin (IPDG) to couple the multiscale basis functions that  contain important heterogeneous media information. The construction of efficient multiscale basis functions starts with extracting dominant modes of carefully defined spectral problems to represent important media feature, which is followed by solving a constraint energy minimization problems.  Then a Petrov-Galerkin projection and systematization onto the coarse grid is applied. As a result,  an explicit and energy conserving  scheme is obtained for fast online simulation. The method exhibits both coarse-mesh and spectral convergence as long as one  appropriately chose the oversampling size. We rigorously analyze the stability and convergence of the proposed method. Numerical results are provided to show the performance of the multiscale method and confirm the  theoretical results.}

\vspace{1ex}
\section{Introduction}
Viscoelasticity and the anisotropy of energy bodies have always been the current research hotspots in multiscale exploration seismology and geophysics. This is because the anisotropy of viscoelasticity and energy body will not only cause the attenuation of seismic wave energy, but also produce velocity anisotropy and dispersion. The existence of these properties will further affect the accuracy of inversion and subsequent imaging. Therefore, the study of the characteristics of wave propagation in viscoelastic anisotropic media can provide valuable information for the subsequent processing of seismic wave data. Besides, the study on wave propagation through fracture media relies on elastic model \cite{chung2016generalized2}. In the past several decades, numerical approaches have been proposed  to study the elastic wave propagation through complicated geological models including the finite difference method (FDM) \cite{claerbout1985imaging, saenger2000modeling, virieux1986p}, finite element method (FEM) \cite{komatitsch2002spectral, drake1989finite, chung2006optimal}, finite volume method (FVM) \cite{eymard2000finite} and pseudospectral method \cite{fornberg1990high} etc. Besides, some improvement strategies of the FDM like the staggered grid \cite{madariaga1976dynamics}, rotating staggered grid \cite{PeiZhenglin2004Numerical} and variable order appeared finite difference format \cite{SongLiwei2017Variable} were also show some success.

Among these methods, the  FEM draw more and more attentions especially in the past two decades due to its  ability  for handling unstructured domain, which is very important for global seismology.
 Some of the earliest attempts to use FEM to solve the wave equation were the traditional continuous Galerkin Finite Element Method (CG-FEM) and its variants the spectral element method (SEM) \cite{ hughes2000continuous, hughes2006multiscale, karakashian1999space,komatitsch1999introduction,komatitsch2000simulation}. But CG-FEM cannot be used to handle the discretization of grids composed of different types of elements, unqualified grids or hanging nodes. These problems can be naturally tackled by the discontinuous Galerkin finite element method(DG-FEM) which is originally developed for the transmission equation \cite{reed1973triangular} and the elliptic partial differential equation \cite{arnold2002unified, riviere1999improved, wheeler1978elliptic}. The DG-FEM enjoys the advantages of the FEM and the FVM, and has been developed on the basis of the two methods, while avoiding the shortcomings of both. 

Regardless of the complexity of the implementation, the above methods suffer from a common issue which is the potential  high computational cost of when solving the wave equation in complicated large scale models. One way to relief this heavy computational burden is adopting the multiscale method. The multiscale method aims to solve the problem in  coarse grid so that  the degrees of freedom can be significantly reduced and thus allow fast simulation.
However, direct simulation on coarse grid with traditional finite element basis functions will lead to large error and this motives the development of using 
special basis functions to replace the polynomial. In \cite{hou1997multiscale}, the  multiscale  finite element method (MsFEM) was proposed for solving highly heterogeneous problem on coarse grid with carefully constructed multiscale basis functions. These multiscale bases are obtained via solving local small scale problems and thus contain important media features, therefore it can yield more accurate coarse-grid solution than using polynomials. Although the MsFEM has  successfully been applied for many practical problems such as flow simulation \cite{chen2003mixed,efendiev2006accurate}, it fails to deal with arbitrary media and 
lacks of accuracy in some scenarios since only one multiscale basis function was constructed in each local domain, and this leads to the development of the 
 generalized multiscale finite element method (GMsFEM) \cite{efendiev2013generalized}. 
In GMsFEM, multiple multiscale basis functions were computed to enrich the basis space in MsFEM and thus improve the accuracy of coarse-grid simulation. 
The GMsFEM has proved to be an efficient and reliable local model reduction technique for many applications \cite{gao2015generalized, fu2017fast,  chung2016mixed, chung2016mixed2, fu2019local,chung2020generalized,chung2019adaptive,chung2018space,li2019generalized}.  
Recently,  the constrained energy minimization generalized multiscale finite element method (CEM-GMsFEM) \cite{chung2018constraint, chung2018nonlinear,Wang2022LocalMM} was proposed to further improve the accuracy of the GMsFEM. The convergence of CEM-GMsFEM depends on the coarse grid size and the spectral value with the least number of basis functions. Initially, the CEM-GMsFEM is only based on CG coupling, but later was extended to the DG formulation \cite{cheung2021explicit}. The DG coupling allows the basis functions to be discontinuous and thus more suitable for media with sharp discontinuity or fracture \cite{chung2016generalized}.

In this paper, the CEM-GMsFEM is used to solve the elastic wave equations within the framework of the Internal Penalty Discontinuity Galerkin (IPDG) method. The method is based on two key components namely the multiscale test basis function and the multiscale experimental basis function.  The first multiscale basis functions are the multiscale test basis functions, which uses an energy reduction technique to select the eigenvectors corresponding to the first few eigenvalues after an incremental ordering of the eigenvalues. This can also be considered as a local spectral problem defined in a coarse block to determine the multiscale basis function. The second multiscale basis functions are obtained from the first component by means of an orthogonality constraint and can also be used for a coarse-scale representation of the numerical solution. Due to the difference between the multiscale test function space and the test function space, we use a finite difference method with a truncation error of second order for time, which results in a technique for localizing the multiscale test basis functions over a coarse oversampling region. The main innovations of this paper are the following: 
\begin{enumerate}
\item  In the study of the anisotropic elastic wave equation, two orthogonality multiscale function spaces are constructed on the discretization of the IPDG method.
\item  A higher-order difference method is used for the temporal discretization of wave propagation in the CEM-GMsFEM.
\item  The CEM-GMsFEM was shown to be locally coarse-grid model stable for the anisotropic elastic wave equation, and yields good error estimates between the coarse- and fine-grid solutions.
\end{enumerate}

Finally, we contextualized our contributions of proving process. We devised one time-step discrete total energy formula that exactly indicates the sum of our method's energy at each distinct moment. We can deduce a global estimate in \autoref{thm1} from the estimates between the fine-scale solution and the coarse-scale solution in Lemma \autoref {lemma1}  and  Lemma \autoref {lemmaa} , which is the highly required finding for assessing various multiscale approaches \cite{gao2015generalized,gavrilieva2020generalized}. By employing oversampling, we increase the accuracy of the study and improve their outcome \cite{chung2014generalized1}. To the best of our knowledge, our technique, on the other hand, is incredibly efficient since it is explicit and does not need inverting any matrices while time marching and the results reported in \autoref{theorem3} are the first of their kind. These theoretical findings provide light on the process underlying the efficiency of elastic wave equations in anisotropic media utilizing CEM-GMsFEMs with DG form.

The structure of this article is as follows. In Section \ref{sec:Preliminaries}, we will introduce the concept of the grid, as well as the basic discretization details, such as the DG-FEM space and the discontinuous Galerkin formula on the coarse grid. The details of the proposed method will be introduced in Section \ref{sec:Construction of multiscale basis function}. This method will be analyzed in Section \ref{sec:Stability and Convergence Analysis}. we will provide the numerical results in Section \ref{sec:Numerical results}. Finally, conclusions will be given in Section \ref{sec:Conclusion}.

\section{Preliminaries}\label{sec:Preliminaries}


We consider the following elastic wave equation in the domain $\Omega \subset \mathbb{R}^{2} $ :
\begin{equation}
    \rho \partial_{t}^{2}{u}=\text{div}\left ( \sigma\left(u\right)\right )+f,\  \text{in} \ \Omega,\label{a}
\end{equation}
where $f =f\left(x,t\right)$ is the given source term. The elastic wave in domain $ \Omega$  is described by the displacement field $u=u\left ( x,t \right )$ and the problem is subject to  a homogeneous Dirichlet boundary condition $u=0$ on $[0, T] \times \partial \Omega$ with initial conditions $$u\left(x, 0\right)=u_{0}\left(x\right)\  \text{and}\  u_{t}\left(x, 0\right)=u_{1}\left(x\right)\ \text{in}\ \Omega.$$
 In (\ref{a}), $\rho $ is the density of the media, and the media in this paper refers to the elastomer. 
\begin{equation}
    {\sigma}=C:{\varepsilon}
\end{equation}
 is the constitutive equation of  anisotropic elastomer, where $\sigma =\sigma \left (u \right )$ represents stress tensor,  $ \epsilon =\epsilon \left ( u \right )$ represents strain tensor, $ C=C_{i,j,p,q}\left ( x \right ) $ is stress where $i,j,p,q=1$ or  $3 $ for two-dimensional plane which can describe the elastic wave propagation in anisotropic media with symmetry up to hexagonal anisotropy with tilted symmetry axis in the $x_{1}-x_{3}$ plane (transversely isotropy with tilted axis, TTI), and monoclinic anisotropy (assuming the symmetry plane is the $x_{1}-x_{3}$ plane).
Then the elastic coefficient matrix is
\begin{equation}
    C=\begin{bmatrix}
C_{11} & C_{13} & 0\\ 
 C_{31}&  C_{33}& 0\\ 
 0& 0 & C_{55}
\end{bmatrix},
\end{equation}
where  $C_{I J} $ are components of the fourth-order elasticity tensor $C $ in Voigt notation.

We assume that  there exist positive constants $0<c_{0} \leq c_{1}$ such that for a.e. $x \in \Omega,$  $ C\left(x\right)$ is a positive definite matrix with $$ c_{0} \leq \lambda_{\min }(C\left(x\right)) \leq \lambda_{\max }\left(C\left(x\right)\right)  \leq c_{1} ,$$
where $\lambda_{\min }\left(C\left(x\right)\right)$ and $\lambda_{\max }\left(C\left(x\right)\right)$ are the minimum and the maximum eigenvalues of $C\left(x\right)$. Let $\epsilon=\frac{1}{2}\left[{\rm grad} u+\left({\rm grad} u\right)^{{\mathrm T}}\right]$ and ${\rm grad} u=\left ( \frac{\partial u_{i}}{\partial x_{j}} \right )_{1\leq i,j\leq 2}$. Moreover, we can write 
\begin{equation}
    \epsilon _{ij}\left ( u \right )=\frac{1}{2}\left ( \frac{\partial u_{i}}{\partial x_{j}}+\frac{\partial u_{j}}{\partial x_{i}}\right), 1\leq i,j\leq 2.
\end{equation}


We are now going to introduce some notions of coarse and fine meshes. 
We define $\mathcal{T}^{H}$ being each element in the coarse grid block as the matching partition of the domain $\Omega$,  and take the symbol $\mathcal{T}^{H}$ as the coarse grid, $H$ as the coarse grid size. The total number of vertices of the coarse mesh is recorded as $N$, and $N_{1}$ is the total number of coarse blocks. In addition,  we call $\mathcal{T}^{h}$ the fine grid, that is, the uniform refinement of the coarse grid, where  $h>0$ is the fine grid size. We assume that the wave equation's fine-mesh discretization offers an accurate approximation of the solution. The standard  bilinear finite element scheme is used for the fine scale solver.

The DG approximation space $V$ is a finite-dimensional function space consisting of a space of coarse-scale locally conforming piecewise bilinear fine-grid basis functions, i.e. it is continuous within the coarse block. In general, however, it is discontinuous at the edges of the coarse grid. Let $u_{h}\in V$ be the solution of \eqref{a}, then it satisfies
\begin{equation}
     \rho s\left( \partial _{t}^{2} u_{h},  v     \right)+a_{\text{DG}}\left (  u_{h}, v \right )=s\left( f,  v     \right), \  \forall\ v \in V,\label{adg}
\end{equation}
where the bilinear form $a_{\text{DG}}\left (  u, v \right )$ is defined as 
\begin{equation}
\begin{aligned}
a_{\text{DG}}\left (  u, v \right )&=\sum_{K\in \mathcal{T} ^{H}}\int_{K} \sigma \left (  u \right ):\epsilon \left ( v \right )\text{dx}\\
&-\sum_{E\in \mathcal{E}^{H}}\int_{E}\left ( \left \{ \sigma\left( u\right) \right \} :\underline{{{[\![ v]\!]}} }+ \underline{{{[\![ u]\!]}}}:  \left \{  \sigma\left( v\right) \right \} \right)\text{ds}\\
&+\sum_{E\in \mathcal{E} ^{H}}\frac{\gamma }{h}\int_{E}\left (\underline{{{[\![ u]\!]}}}:\left \{  C \right \}:\underline{{{[\![ v]\!]}}} +{[\![ u]\!]}\cdot \left \{ D\right \}\cdot {[\![ v]\!]}\right )\text{ds},\label{equ:bilinear}
\end{aligned}
\end{equation}
with 
\begin{equation}
    D=\begin{bmatrix}
C_{11} & 0 & 0\\ 
 0& C_{22} &0 \\ 
 0& 0 & C_{33}
\end{bmatrix}
\end{equation}
being the diagonal matrix made up of $C$'s diagonal entries,  and 
\begin{equation}
s\left( u,  v     \right)= \int_{\Omega } u v \text{dx}.
\end{equation}

 In (\ref{equ:bilinear}),   $[\![ v]\!]$  and \underline{{{[\![v]\!]}}} are the vector jump and the matrix jump respectively and are defined as
\begin{equation}
    {{[\![ v]\!]}}=v^{+}-v^{-},\  \underline{{{[\![v]\!]}}}=v^{+} \otimes n^{+}+v^{-} \otimes n^{-},
\end{equation}
where $'+'$ and $'-'$ respectively represent the values on two adjacent cells $K^+$, $K^-$ sharing the coarse grid edge, and $n^+$ and $n^-$ are the unit outward normal vector on the boundary of $K^+$ and $K^-$. We can also define the average of a tensor ${\sigma} $ as 
\begin{equation}
    \left\{\sigma\right\} =\frac{1}{2}\left( \sigma ^{+}+ \sigma ^{-}\right).
\end{equation}
The initial data in problem (\ref{adg}) are obtained by: find $ u_{h}\left(\cdot, 0\right) $ and $ {u}'_{h}\left(\cdot, 0\right) $ such that for all $v\in V$,
\begin{equation}
\begin{aligned}
s\left(u_{h}\left(\cdot, 0\right), v\right) &=s\left(u_{0}, v\right) \\
s\left({u}'_{h}\left(\cdot, 0\right), v\right)&=s\left(u_{1}, v\right),
\end{aligned}
\end{equation}
where ${u}'_{h}\left(\cdot, 0\right)=\partial_t \left( u_{h}\right)\left(\cdot, 0\right).$


\section{Construction of the multiscale basis functions }\label{sec:Construction of multiscale basis function}
For the GMsDGM, we have  following steps in the multiscale simulation algorithm:
\begin{enumerate}
\item Generate coarse grid and local domain;
\item Construct a multiscale basis function by solving the local eigenvalue problem of each local domain;
\item Construct the projection matrix $R$ from the fine grid to the coarse grid and construct the fine grid system;
\item The solution of the reduced-order model and the reconstruction of the fine-grid solution.
\end{enumerate}
Taking the local domain as the coarse cell $K_{j},$ then for the GMsDGM, we construct two local multiscale spaces (boundary and interior) by solving the local eigenvalue problem on each coarse grid cell $K_{j}.$ 
The multiscale basis functions vary depending on the definition of the local spectrum problem.

\subsection{Multiscale Model}
\textbf{Part\ 1. Multiscale test functions}\\
In this part, we complete the first step of calculating multiscale basis functions, and perform basis function calculations on fine-scale grids. To do that, the restrictions of $V$ on $K_{j}$ is denoted by $V\left(K_{j}\right)$ then it  performs a dimension reduction through a spectral problem: find $\left ( \psi _{i}^{j},\lambda  _{j}^{i} \right )\in V\left ( K_{j} \right )\times \mathbb{R} $ such that
\begin{equation}
    a^{j}\left ( \psi _{i}^{j}, v\right )= \frac{\lambda _{j}^{i}}{H^2}s^{j}\left ( \psi _{i}^{j}, v \right ),\ \forall \ v\in V\left ( K_{j} \right ), \label{equ:eigen}
\end{equation}
where  $H$ denotes the length of the edge $E$ and $\lambda_{j}^{i} $ is the eigenvalue in order $i$-\text{th} which is in ascending for each $i \in \left\{1,...,N\right\}$ in $K_j$.   In (\ref{equ:eigen}),  $a^{j}$ and $s^{j}$ are non-negative and positive symmetric definite bilinear operators defined on $V\left(K_{j}\right) \times V\left(K_{j}\right)$.
We remark that 
\begin{equation}
    a^{j}\left ( w,v \right )=\int_{K_{j}} \sigma \left(w\right) :  \epsilon \left(v\right) \text{dx},
\end{equation}
\begin{equation}
     s^{j}\left( \psi _{i}^{j},  \psi _{k}^{j}\right)=\int_{K_j}   \psi _{i}^{j}  \psi _{k}^{j}   \text{dx}   =   \left\{\begin{matrix}
1 ,& i=k\\ 
0, & i\neq k
\end{matrix}\right.
\end{equation}
where $1\leq i,k \leq g_j.$
Then we use $g_{j}$ eigenvalue functions to construct our local auxiliary multiscale space 
$$V_{\text{aux}}^{j}=\text{span}\left \{ \psi _{i}^{j}:1\leq i\leq g_{j} \right \}.$$
Based on these local spaces, we define the global auxiliary multiscale space $V_{\text{aux}}=\bigoplus_{j=1}^{N}V_{\text{aux}}^{j},$ and the bilinear form $s^{j}$ defined as an inner product $$
s\left ( u,v\right )=\sum_{j=1}^{N}s^{j}\left ( u,v\right ),\  \forall \ u,v \in V_{\text{aux}}.$$
Finally, we define a orthogonal mapping $\pi_j: V\rightarrow V_{\text{aux}}$ such that 
$$\pi_{j}\left ( v\right )=\sum_{i=1}^{g_{j}}s^{j}\left ( v,\psi _{i}^{j} \right )\psi _{i}^{j},\  \forall \ v\in V.$$ Moreover, we let $ \pi=\sum_{j=1}^{N}\pi_{i}.$

\textbf{Part\ 2. Multiscale trial functions}\\
We give the definition of $\psi _{i}^{j}\text{-orthogonal}$ as following: for $\psi _{i}^{j} \in V\left(K_j\right),$ there exists $\phi$  in $V$ such that $\pi\left(\phi\right)=\psi_{i}^{j}.$
Next we define the global multiscale basis function as 
\begin{equation}
\phi _{i}^{j}:=\text{argmin}\left \{ a_{\text{DG}}\left ( \phi ,\phi  \right ):\phi \in V\  \text{is}\ \psi _{i}^{j}\text{-orthogonal}\right \}.
\end{equation}
By introducing a Lagrange multiplier, the minimization problem is equivalent to the following problem: find $\left ( \phi _{i}^{j},\eta_{i}^{j} \right )\in\left (V\times V_{\text{aux}}\right ),$ such that 
\begin{equation}
    \begin{aligned}
    a_{\text{DG}} \left ( \phi _{i}^{j},\phi  \right )+s\left ( \phi ,\eta _{i} ^{j}\right )&=0,\ \forall \ \phi \in V,\\
    s\left ( \phi _{i}^{j}-\psi _{i}^{j} ,\eta \right )&=0,\ \forall \ \eta \in V_{\text{aux}}.
    \end{aligned}
\end{equation}
Define the localized multiscale trial space as:
$$ V_{\rm{cem}}=\text{span}\left\{\phi_{i}^{j}: 1 \leq i \leq g_{j}, 1 \leq j \leq N_1\right\}.$$
Now we denote $K_{j,r}$ as an oversampling domain formed by enlarging $K_{j}$ by $r$ coarse grids. 
For each auxiliary function $\psi _{i}^{j}\in V_{\text{aux}}$, we solve the multiscale basis function possessing the property of constraint energy minimization problem: find $\phi _{i,r}^{j}\in V\left ( K_{j,r} \right ),$ such that 
\begin{equation}
   \phi _{i,r}^{j}={\rm argmin}\left \{ a_{\text{DG}} \left ( \phi ,\phi  \right ):\phi \in V\left ( K_{j,r} \right ){\rm is}\ \psi _{i}^{j}\text{-orthogonal} \right \}. 
\end{equation}
By introducing a Lagrange multiplier, the minimization problem is equivalent to the following problem: find $\left ( \phi _{i,r}^{j},\eta _{i,r}^{j} \right )\in\left ( V\left ( K_{j,r} \right )\times V_{\text{aux}}\left ( K_{j,r} \right )\right ),$ such that
\begin{equation}
    \begin{aligned}
     a_{\text{DG}} \left ( \phi _{i,r}^{j},\phi  \right )+s\left ( \phi ,\eta _{i,r} ^{j}\right )&=0,\ \forall \ \phi \in V\left ( K_{j,r} \right ),\\
    s\left ( \phi _{i,r}^{j}-\psi _{i}^{j} ,\eta \right )&=0,\ \forall\ \eta \in V_{\text{aux}}\left ( K_{j,r} \right ).
    \end{aligned}
\end{equation}
The localized multiscale trial basis functions then used to define the localized multiscale trial space, i.e.
\begin{equation}
V_{\rm{cem}}^{r}=\operatorname{span}\left\{ \phi _{i,r}^{j}: 1 \leq i \leq g_{j}, 1 \leq j \leq N_1\right\}.
\end{equation}

\subsection{Fully Discretization}
  The multiscale solution $u_{H}$ is defined as the solution of the following problem: find $u_{H}\in V_{\rm{cem}}^{r}$, such that 
\begin{equation}
   \int_{\Omega } \rho \partial _{t}^{2} u_{H}\cdot v \text{dx}+a_{\text{DG}}\left ( u_{H}, v \right )=\int_{\Omega } fv \text{dx}\label{ms},\ \forall \ v\in V. 
\end{equation}
First by the definition of $\pi$, for $v \in V_{\rm{cem}}$, we note that 
\begin{equation}
s\left(v-\pi\left(v\right), w\right)=0,\forall\ w \in V_{\text{aux}}  \label{equ:111}
\end{equation}
and
\begin{equation}
a_{\text{DG}}\left(\psi, v-\pi\left(v\right)\right)=0,\forall\  \psi \in V_{\rm{cem}} . \label{equ:222}
\end{equation}

Next we introduce some related methods about time discretization, and let $N_{T}$ be the number of time steps in the time grid, $\tau =T/N_{T}$ be the time step size,\ $f^{m}$\ be defined as the value of function $f$ at the time $t_{m}=m\tau$, and $u_{H}^{m}$ represents the approximate value of $u_H \left(\cdot,t_{m} \right)$. By 
\begin{equation*}
    u ^{m+1}_H =u ^{m}_H+\tau \partial _{t }u^{m}_H+\frac{ \tau ^{2}}{2!}\partial _{t }^{2}u ^{m}_H+\frac{ \tau ^{3}}{3!}\partial _{t }^{3}u ^{m}_H+...
\end{equation*}
and $$   u ^{m-1}_H =u ^{m}_H-  \tau \partial _{t }u ^{m}_H+\frac{  \tau ^{2}}{2!}\partial _{t }^{2}u ^{m}_H-\frac{ \tau ^{3}}{3!}\partial _{t }^{3}u^{m}_H+...,$$
we simply have central finite difference in second-order precision as 
\begin{equation}
\begin{aligned}
    \ddot{\delta }_{H}&=\frac{u _{H}^{m+1}-2u _{H}^{m}+u_{H}^{m-1}}{\tau ^{2}},\\
    \dot{\delta}_{H}&=\frac{u_{H}^{m+1}-u _{H}^{m-1}}{2\tau}.
    \end{aligned}
\end{equation}
Combining (\ref{ms}), (\ref{equ:111}) and  (\ref{equ:222}), the coarse-scale model which reads: for $m \geq 1$, find $u_{H}^{m+1} \in V_{\rm{cem}}$  such that
\begin{equation}
{\pi}'\left(\frac{u_{H}^{m+1}-2 u_{H}^{m}+u_{H}^{m-1}}{\tau^{2}}, v\right)+a_{D G}\left(u_{H}^{m}, v\right)={\pi}'\left(f^{m}, v\right) ,\forall\ v \in V_{\text{aux}}, \label{equ:coarse}
\end{equation}
where the initial data is projected onto the finite element space $V_{\rm{cem}}$ by: find $u_{H}^{0}, u_{H}^{1} \in V_{\rm{cem}}$ such that  for $ v \in V_{\rm{cem}}$,
\begin{equation}
\begin{aligned}
s\left(u_{H}^{0}, v\right)&=s\left(u_{0}, v\right) \\
s\left(u_{H}^{1}, v\right)&=s\left(u_{0}+\tau u_{1}+\frac{\tau^{2}}{2} f^{0}, v\right)-\frac{\tau^{2}}{2} a_{\text{DG}}\left(u_{H}^{0}, v\right).
\end{aligned}
\end{equation}
In (\ref{equ:coarse}), we use the second order central difference approximation for the time derivative. We have also used the operator  $\pi'$ and  $\pi'$  is defined as
\begin{equation}
   {\pi}' \left(w,u\right)=s\left(\pi\left(w\right),\pi\left(u\right)\right).
\end{equation}
 This ensures that the mass matrix is block diagonal, which allows efficient time stepping \cite{cheung2021explicit}.

Finally, our localized coarse-grid model reads: for $ m>1$, find $u_{H}^{m+1}\in V_{\rm{cem}}^r$ such that
\begin{equation}
    {\pi}' \left(\frac{u_{H}^{m+1}-2u_{H}^{m}+u_{H}^{m-1}}{\tau^{2}},u \right)+a_{\text{DG}}\left ( u_{H}^{m},u \right )={\pi}'  \left ( f^{m},u \right ),\ \forall \ u\in V_{\rm{cem}}^r\label{adgtime},
\end{equation}
Considering initial data in the finite element space, we could find $u_{H}^{0},u_{H}^{1}\in V_{\rm{cem}}^r$ such that for all $v\in V_{\rm{cem}}^r$,
\begin{equation}
\begin{aligned}
   s \left ( u _{H}^{0}, v\right )&=s\left ( u _{0}, v\right ),\\
s\left ( u _{H}^{1}, v\right )&=s\left ( u _{0}+\tau u_{1}+\frac{\tau ^{2}}{2}f^{0}, v\right )-\frac{\tau ^{2}}{2}a_{\text{DG}}\left (u _{H}^{0},v \right ).
\end{aligned}
\end{equation}

\section{Stability and Convergence Analysis}\label{sec:Stability and Convergence Analysis}
In this section, we analyze the stability and convergence of \eqref{adgtime}. 
For our analysis, we define the following DG norm
\begin{equation}
    \left \|  v  \right \|_{\text{DG}}^{2}=\sum_{K\in\mathcal{T} ^{H}}\int_{K} \epsilon \left ( v  \right ):C:\epsilon \left ( v  \right ) \text{dx}+\frac{\gamma }{h}\sum_{E\in \mathcal{E}^{H}}\int_{E}\underline{{[\![ v]\!] }}:C:\underline{{[\![ v]\!] } }+[\![ v]\!]\cdot D\cdot [\![ v]\!]\text{ds}.
\end{equation}
Firstly we choose the reflection $P:=I-\pi ,$ for $w\in V,$
then we have 
\begin{equation}
    \left \| Pw \right \|_{s}\leq 2c_{2}^{-1}H^{2}a_{\text{DG}}\left ( w,w \right ),\label{39}
\end{equation}
where $c_{2}=\underset{i,j\leq N_1}{\rm{max}}\lambda _{g_{i}+1}^{j}$.

We define a discrete total energy which is related to the stability and convergence of our method. Given a sequence of status $ u =\left ( u ^{m} \right )_{m=0}^{N_{T}}$, and we define the discrete total energy at $t=t_{m+1}$ by
\begin{equation}
     Q^{m+1}\left ( u  \right )=\left \| \pi \left ( \frac{u ^{m+1}-u ^{m}}{\tau } \right ) \right \|_{s}^{2}
     -\left ( \frac{\tau }{2} \right )^{2}a_{\text{DG}}\left ( \frac{u ^{m+1} -u ^{m}}{\tau },\frac{u ^{m+1} -u ^{m}}{\tau }\right )+a_{\text{DG}}\left ( u^{m},u ^{m} \right ), \label{equ:state}
\end{equation}
which is greater than zero under stable condition.

\begin{lemma}
There exists a constant $c_{3}>0$ such that 
\begin{equation}
a_{\text{DG}}\left ( u ^{m+1}-u ^{m}, u ^{m+1}-u ^{m}\right )  \leq   \frac{c_{1}\tau ^{2}}{c_{3}H^2}     \left \| \pi \left ( \frac{u ^{m+1}-u ^{m}}{\tau } \right ) \right \|_{s}^{2},\label{52}
\end{equation}
where  $u^m, u^{m+1} \in V_{\rm{cem}}^r$ is the sequence in  (\ref{equ:state}).
 Moreover, if  $\frac{c_{1}\tau ^{2}}{4c_{3}H^2}<\eta, $ where $0 <\eta <1 $, then we have the following formulation:
\begin{equation}
     Q^{m+1}\left ( u  \right )\geq \left ( 1- \frac{c_{1}\tau ^{2}}{4c_{3}H^2} \right )\left \| \pi \left ( \frac{u ^{m+1}-u ^{m}}{\tau } \right ) \right \|_{s}^{2}+a_{\text{DG}}\left ( u ^{m},u ^{m} \right )\geq 0. \label{equ:27}
\end{equation}
\end{lemma}
\begin{proof}

From \cite{chung2018constraint}, there exists $D>0$ such that for any $ u_{\text{aux}} \in V_{\text{aux}}$, there exists a function $v \in C^{0}(\Omega) \cap V$ such that
\begin{equation}
\pi\left(v\right)= u_{\text{aux}}, \quad\|v\|_{\text{DG}}^{2} \leq D c_1 H^{-2}\left\| u_{\text{aux}}\right\|_{s}, \quad \operatorname{supp}\left(v\right) \subseteq \operatorname{supp}\left( u_{\text{aux}}\right) .
\end{equation}
Again in \cite{chung2018constraint},  there exists $c_3>0$ such that
\begin{equation}
\|\pi\left(v\right)\|_{s}^{2} \geq c_3 c_{1}^{-1} H^{2} a_{\text{DG}}\left(v, v\right),\  \forall\  v \in V_{\rm{cem}}^{r}.
\end{equation}
Taking $v=  \frac{u ^{m+1}-u ^{m}}{\tau } ,$  then we directly proof (\ref{52}). Combining (\ref{52}) and (\ref{equ:state}), we have (\ref{equ:27}).
\end{proof}
Remark: In the rest of the paper, we assume  $\frac{c_{1}\tau ^{2}}{4c_{3}H^2}<\eta, $ where $0 <\eta <1 $.

\subsection{Stability Analysis}
The stability analysis will start with the following lemma.
\begin{lemma}\label{lemma1}
For $m\geq 1,$ given $z^{m}\in L^{2}\left ( \Omega  \right ),\ u ^{m-1}, \ u ^{m}\in V_{\rm{cem}}^{r},$  we suppose $ u ^{m+1}\in V_{\rm{cem}}^{r}$ as the solution of the following problem:
\begin{equation}
    {\pi}'\left ( \frac{u ^{m+1}-2u ^{m}-u ^{m-1}}{\tau ^{2}},u \right )+a_{\text{DG}}\left ( u ^{m},u \right )=s\left ( z^{m},u \right ),\  \forall\ u\in V_{\rm{cem}}^{r}\label{lemma53}.
\end{equation}
Then we have 
\begin{equation}
    Q^{m+1}\left(u \right)=Q^{1}\left (u \right)+\tau \sum_{k=1}^{m}s\left ( z^{k},\frac{u^{k+1}-u ^{k}}{\tau } \right ),
\end{equation}
where $$Q^{1}\left (u \right)=\left \| \pi \left ( \frac{u ^{1}-u ^{0}}{\tau } \right ) \right \|_{s}^{2}
     -\left ( \frac{\tau }{2} \right )^{2}a_{\text{DG}}\left ( \frac{u ^{1} -u ^{0}}{\tau },\frac{u ^{1} -u ^{0}}{\tau }\right )+a_{\text{DG}}\left ( u^{0},u ^{0} \right ).$$
Moreover, we can obtain the estimate of $Q^{m+1}\left(u\right )$ as 
\begin{equation}
   Q^{m+1}\left(u\right )\leq  \frac{16}{15}\left(Q\left(u\right)+4\tau ^{2}\left(1- \frac{c_{1}\tau ^{2}}{4c_{3}H^2} \right)^{-1}\left ( \sum_{k=1}^{m} \left \| \pi \left ( z^{k} \right ) \right \|_{s} \right )^{2}+\frac{2H^{2}}{c_2}\left ( D^{m} \right )^{2}\right),
\end{equation}
where
\begin{equation}
    D^{m}=2\left(\left \| Pz^{1} \right \|_{s}+\tau \sum_{k=1}^{m}\left \| P\left ( \frac{r ^{k+1}-r ^{k}}{\tau } \right ) \right \|_{s}+\left \| Pz^{m} \right \|_{s}\right). 
\end{equation}
\end{lemma}

\begin{proof}
Let $u=\frac{u^{m+1}-u ^{m-1}}{2\tau }$ in (\ref{lemma53}), then we have
\begin{equation*}
    \begin{aligned}
    {\pi}'\left ( \frac{u ^{m+1}-2u ^{m}-u ^{m-1}}{\tau ^{2}},\frac{u^{m+1}-u ^{m-1}}{2\tau }\right )
    &+a_{\text{DG}}\left ( u ^{m},\frac{u^{m+1}-u ^{m-1}}{2\tau } \right )\\
    &=s\left ( z^{m},\frac{u^{m+1}-u ^{m-1}}{2\tau }\right ),
    \end{aligned}
\end{equation*}
that is 
\begin{equation*}
    \begin{aligned}
   & \frac{1}{\tau }{\pi}'\left ( \frac{u ^{m+1}-u ^{m}-\left(u ^{m}-u ^{m-1}\right)}{\tau },\frac{u ^{m+1}-u ^{m}+u ^{m}-u ^{m-1}}{2\tau} \right )\\
   &\quad +\frac{1}{2\tau }\left(a_{\text{DG}}\left ( u ^{m},u ^{m+1}\right )-a_{\text{DG}}\left ( u ^{m},u ^{m-1}\right )\right )
    =s\left ( z^{m},\frac{u ^{m+1}-u ^{m-1}}{2\tau}\right ).
    \end{aligned}
\end{equation*}
Then we get
\begin{equation}
\begin{aligned}
  \frac{1}{2\tau }&\left ( \left \| \pi \left ( \frac{u ^{m+1}-u ^{m}}{\tau } \right ) \right \|_{s}^{2}- \left \| \pi \left ( \frac{u ^{m}-u ^{m-1}}{\tau } \right ) \right \|_{s}^{2}\right)\\
&+\frac{1}{2\tau }\left (a_{\text{DG}} \left ( u ^{m+1} ,u ^{m}\right )-a_{\text{DG}} \left ( u ^{m},u ^{m-1} \right )\right )=s\left ( z^{m} ,\frac{u ^{m+1}-u ^{m-1}}{2\tau }\right ).
\end{aligned}
\end{equation}
We also get the following equations by the definition,
\begin{equation}
a_{\text{DG}}\left(u^{m+1}, u^{m}\right)=a_{\text{DG}}\left(\frac{u^{m+1}+u^{m}}{2}, \frac{u^{m+1}+u^{m}}{2}\right)-\frac{\tau^{2}}{4} a_{\text{DG}}\left(\frac{u^{m+1}-u^{m}}{\tau}, \frac{u^{m+1}-u^{m}}{\tau}\right),
\end{equation}
and 
\begin{equation}
a_{\text{DG}}\left(u^{m}, u^{m-1}\right)=a_{\text{DG}}\left(\frac{u^{m}+u^{m-1}}{2}, \frac{u^{m}+u^{m-1}}{2}\right)-\frac{\tau^{2}}{4} a_{\text{DG}}\left(\frac{u^{m}-u^{m-1}}{\tau}, \frac{u^{m}-u^{m-1}}{\tau}\right).
\end{equation}
Hence we have 
\begin{equation}
\begin{aligned}
    Q^{m+1}\left ( u  \right )-Q^{m}\left ( u  \right )
    =\tau s\left ( z^{m},\frac{u^{m+1}-u ^{m-1}}{2\tau } \right )
\end{aligned}
\end{equation}

We rewrite the right hand side of the above equation as 
\begin{equation}
\begin{aligned}
s\left ( z^{m},\frac{u^{m+1}-u ^{m-1}}{2\tau } \right )
&= {\pi}'\left (z^{m},\frac{u^{m+1}-u ^{m-1}}{2\tau }\right ) +s\left ( Pz^{m}, \frac{u^{m+1}-u ^{m-1}}{2\tau }\right )\\
&=\frac{1}{2}\left ( {\pi}'\left (z^{m},\frac{u^{m+1}-u ^{m}}{\tau }\right )+{\pi}'\left ( z^{m},\frac{u^{m}-u ^{m-1}}{\tau }\right ) \right )\\
&\qquad +\frac{1}{\tau }s\left ( Pz^{m} ,P\frac{u^{m+1}+u^{m}}{2}\right )-\frac{1}{\tau }s\left ( Pz^{m} ,P\frac{u^{m}+u^{m-1}}{2}\right ).
\end{aligned}
\end{equation}
Next we obtain
\begin{equation}
    \begin{aligned}
    Q^{m+1}\left ( u \right )
    &=Q\left ( u \right ) + \frac{ \tau}{2} \sum_{k=1}^{m}\left ( {\pi}'\left ( z^{k} ,\frac{u^{k+1}-u ^{k}}{\tau }\right )+{\pi}'\left ( z^{k} ,\frac{u^{k}-u ^{k-1}}{\tau }\right ) \right )\\
    &\qquad+  \sum_{k=1}^{m}s\left ( Pz^{k}, P\frac{u^{k+1}+u^{k}}{2} \right ) - \sum_{k=1}^{m}s\left ( Pz^{k}, P \frac{u^{k}+u^{k-1}}{2}\right )               \\
    &=Q\left ( u \right ) + \frac{ \tau}{2}\sum_{k=1}^{m}\left ( {\pi}'\left ( z^{k} ,\frac{u^{k+1}-u ^{k}}{\tau }\right )+{\pi}'\left ( z^{k} ,\frac{u^{k}-u ^{k-1}}{\tau }\right ) \right )\\
    &\qquad -\tau \sum_{k=1}^{m-1}s\left ( P\left ( \frac{r^{k+1} -z^{k}}{\tau }\right ), P \left ( \frac{u^{k+1} +u^{k}}{2 }\right )\right )\\
    &\qquad+s\left ( Pz^{m} ,P\left ( \frac{u^{m+1} +u^{m}}{2 }\right )\right )-s \left ( Pz^{1} ,P\left ( \frac{u^{1} +u^{0}}{2 }\right )\right ).
    \end{aligned}
\end{equation}
Finally, we use Cauchy-Schwarz inequality and Young's inequality to obtain the following inequality, as 
\begin{equation}
    \begin{aligned}
    Q^{m+1}\left ( u  \right )
&\leq Q\left ( u  \right )+\underset{0\leq k\leq m}{\rm{max}}\left \| \pi \left ( \frac{u ^{k+1}-u ^{k}}{\tau } \right ) \right \|_{s}\cdot \tau \sum_{p=1}^{m} \left \| \pi \left ( z^{p} \right ) \right \|_{s}\\
&\quad +\underset{0\leq k\leq m}{\rm{max}}\left \| P\left ( \frac{u ^{k+1}+u ^{k}}{\tau } \right ) \right \|_{s}\cdot D^{m}\\
&\leq Q\left ( u  \right )+\underset{0\leq k\leq m}{\rm{max}}\frac{1}{16}\left(1- \frac{c_{1}\tau ^{2}}{4c_{3}H^2} \right)\left \| \pi \left ( \frac{u ^{k+1}-u ^{k}}{\tau } \right ) \right \|_{s}^{2}\\
&\quad+4\tau ^{2}\left(1- \frac{c_{1}\tau ^{2}}{4c_{3}H^2} \right)^{-1}\left ( \sum_{p=1}^{m} \left \| \pi \left ( z^{p} \right ) \right \|_{s} \right )^{2}\\
&\quad +\frac{c_2}{8H^2}\underset{0\leq k\leq m}{\rm{max}}\left \| P\left ( \frac{u ^{k+1}+u ^{k}}{2 } \right ) \right \|_{s}^{2}+\frac{2H^{2}}{c_2}\left ( D^{m} \right )^{2},
    \end{aligned}
\end{equation}
where
\begin{equation*}
    D^{m}=2\left(\left \| Pz^{1} \right \|_{s}+\tau \sum_{k=1}^{m}\left \| P\left ( \frac{r ^{k+1}-r ^{k}}{\tau } \right ) \right \|_{s}+\left \| Pz^{m} \right \|_{s}\right). 
 \end{equation*}
Reviewed the (\ref{39}) and (\ref{52}), we have
\begin{equation}
    \begin{aligned}
    Q^{m+1}\left ( u  \right )
    & \leq Q\left(u\right)+\frac{1}{16}\left(1- \frac{c_{1}\tau ^{2}}{4c_{3}H^2} \right)\times \left(1- \frac{c_{1}\tau ^{2}}{4c_{3}H^2} \right)^{-1} \left(Q^{m+1}\left(u\right)-a_{\rm{DG}}\left(u^{m},u^{m}\right)\right)\\
    &\quad +4\tau ^{2}\left(1- \frac{c_{1}\tau ^{2}}{4c_{3}H^2} \right)^{-1}\left ( \sum_{p=1}^{m} \left \| \pi \left ( z^{p} \right ) \right \|_{s}\right)^{2}\\
    &\quad+\frac{c_2 H^{-2}}{8}\times\frac{H^{2}}{2c_2}a_{\rm{DG}}\left(u^{m},u^{m}\right)+\frac{2H^{2}}{c_2}\left ( D^{m} \right )^{2}\\
    &\leq \frac{16}{15}\left(Q\left(u\right)+4\tau ^{2}\left(1- \frac{c_{1}\tau ^{2}}{4c_{3}H^2} \right)^{-1}\left ( \sum_{k=1}^{m} \left \| \pi \left ( z^{k} \right ) \right \|_{s} \right )^{2}+\frac{2H^{2}}{c_2}\left ( D^{m} \right )^{2}\right).
    \end{aligned}
\end{equation}
This proofs the Lemma.
\end{proof}

\begin{theorem}
With the above assumptions and $u_{H} ^{m}, u_{H}^{m+1} $ are solutions of (\ref{adgtime}) and $u_{H}$ is  solution of (\ref{ms}). We have the following stability estimate in this paper.
\begin{equation}
\left \| \pi\left ( \frac{u_{H}^{m+1}-u_{H} ^{m}}{\tau } \right ) \right \|_{s}^{2}+a_{\text{DG}}\left (u_{H}^{m},u_{H}^{m}  \right )\leq M_{1}\left ( Q\left ( u _{H} \right )+M_{2}\tau ^{2}\left ( \sum_{k=1}^{m} \left \| \pi\left ( f^{k} \right ) \right \|_{s}^{2}\right )^{2} \right ),
\end{equation}
where  $M_{1}=\frac{16}{15}\left(1- \frac{c_{1}\tau ^{2}}{4c_{3}H^2} \right)^{-1}$ and $M_{2}=4\left(1- \frac{c_{1}\tau ^{2}}{4c_{3}H^2} \right)^{-1} $.
\end{theorem}
\begin{proof}
By the above all lemma, we can directly get the following inequality.
\begin{equation*}
    \begin{aligned}
     Q^{m+1}\left ( u_{H}  \right )&=\left \| \pi \left ( \frac{u_{H} ^{m+1}-u_{H} ^{m}}{\tau } \right ) \right \|_{s}^{2}\\
     &\qquad -\left ( \frac{\tau }{2} \right )^{2}a_{\text{DG}}\left ( \frac{u_{H} ^{m+1} -u_{H} ^{m}}{\tau },\frac{u_{H} ^{m+1} -u_{H} ^{m}}{\tau }\right )
     +a_{\text{DG}}\left ( u_{H} ^{m},u_{H} ^{m} \right )\\
     &\leq \frac{16}{15}\left(Q\left(u_{H}\right)+4\tau ^{2}\left(1- \frac{c_{1}\tau ^{2}}{4c_{3}H^2} \right)^{-1}\left ( \sum_{k=1}^{m} \left \| \pi \left ( f^{k} \right ) \right \|_{s} \right )^{2}+\frac{2H^{2}}{c_2}\left ( D^{m} \right )^{2}\right).
    \end{aligned}
\end{equation*}
Simplify the above inequality, we see that
\begin{equation*}
    \begin{aligned}
    &\qquad\left(1- \frac{c_{1}\tau ^{2}}{4c_{3}H^2} \right) \left \| \pi \left ( \frac{u_{H} ^{m+1}-u_{H} ^{m}}{\tau } \right ) \right \|_{s}^{2}+\left(1-\frac{16}{15}\times \frac{4H^{2}}{c_2}m\right) a_{\text{DG}}\left ( u_{H} ^{m},u_{H} ^{m} \right )\\
    &\leq \frac{16}{15}\left(Q\left(u_{H}\right)+4\tau ^{2}\left(1- \frac{c_{1}\tau ^{2}}{4c_{3}H^2} \right)^{-1}\left ( \sum_{k=1}^{m} \left \| \pi \left ( f^{k} \right ) \right \|_{s} \right )^{2}\right).
    \end{aligned}
\end{equation*}
Let $c_3\geq \frac{15c_{1}c_{2}\tau ^{2}}{169mH^2}>0$, that is,  
$$\left(1- \frac{c_{1}\tau ^{2}}{4c_{3}H^2} \right)=\max \left\{\left(1- \frac{c_{1}\tau ^{2}}{4c_{3}H^2} \right) ,\left(1-\frac{16}{15}\times \frac{4H^{2}}{c_2}m\right)\right\} ,$$
then we have
\begin{equation*}
    \begin{aligned}
    &\qquad \left(1- \frac{c_{1}\tau ^{2}}{4c_{3}H^2} \right) \left(\left \| \pi \left ( \frac{u_{H} ^{m+1}-u_{H} ^{m}}{\tau } \right ) \right \|_{s}^{2}+ a_{\text{DG}}\left ( u_{H} ^{m},u_{H} ^{m} \right )\right)\\
    &\leq \frac{16}{15}\left(Q\left(u_{H}\right)+4\tau ^{2}\left(1- \frac{c_{1}\tau ^{2}}{4c_{3}H^2} \right)^{-1}\left ( \sum_{k=1}^{m} \left \| \pi \left ( f^{k} \right ) \right \|_{s} \right )^{2}\right).
    \end{aligned}
\end{equation*}
That is
\begin{equation*}
    \begin{aligned}
    &\qquad\left \| \pi \left ( \frac{u_{H} ^{m+1}-u_{H} ^{m}}{\tau } \right ) \right \|_{s}^{2}+ a_{\text{DG}}\left ( u_{H} ^{m},u_{H} ^{m} \right )\\
    &\leq \left(1- \frac{c_{1}\tau ^{2}}{4c_{3}H^2} \right) ^{-1}\times \frac{16}{15}\left(Q\left(u_{H}\right)+4\tau ^{2}\left(1- \frac{c_{1}\tau ^{2}}{4c_{3}H^2} \right)^{-1}\left ( \sum_{k=1}^{m} \left \| \pi \left ( f^{k} \right ) \right \|_{s} \right )^{2}\right),
    \end{aligned}
\end{equation*}
where we let  $M_{1}=\frac{16}{15}\left(1- \frac{c_{1}\tau ^{2}}{4c_{3}H^2} \right)^{-1}$ and $M_{2}=4\left(1- \frac{c_{1}\tau ^{2}}{4c_{3}H^2} \right)^{-1} $.

Hence we complete the proof of stability.
\end{proof}
Remark: In the rest of the paper, we assume $c_3\geq \frac{15c_{1}c_{2}\tau ^{2}}{169mH^2}>0.$

\subsection{Convergence Analysis}
To proof the convergence of the multisacle basis function, we introduce two operators to proceed our convergence analysis.
 \begin{enumerate}
\item [1)] 
 $$E_{h}:L^{2}\left ( \Omega  \right )\rightarrow V$$
 for $\forall e \in L^{2}\left ( \Omega  \right ) $, the image $ E_{h}e \in V$ is defined as
  \begin{equation}
     a_{\text{DG}}\left ( E_{h}e, u \right )=s\left ( e,u \right ),\  \forall \ u\in V . \label{equ:conver1}
 \end{equation}
\item [2)] 
 $$F_{H}:V\rightarrow V^{r}_{\rm{cem}}$$
 as a elliptic projection is defined by 
 \begin{equation}
     a_{\text{DG}}\left ( F_{H}v , u \right )=a_{\text{DG}}\left ( v,u \right ),\ \forall \ u\in V_{\rm{cem}}^r .
 \end{equation}
 where for any $v \in V$, $F_{H}v \in V_{\rm{cem}}^r$ is the image of $v$.
\end{enumerate}

\begin{lemma}
Assume that the size of oversampling region as coarse mesh $r=O\left(\log \left(\frac{\lambda _{\max} \left(C\left(x\right)\right)}{H}\right)\right)$, then there exists a constant $c_4>0$ such that 
\begin{equation}
    \left \| S_{H}E_{h}e \right \|_{\text{DG}}\leq \frac{c_4H}{\sqrt{c_2}}\left \| e \right \|_{s},\ \forall\  e \in L^{2}{\left (\Omega  \right ) }, \label{equ:lemma2}
\end{equation}
where $S_{H}=I-F_{H}.$

Moreover, we get 
\begin{equation}
   \left \| S_{H}E_{h}e \right \|_{s}\leq \frac{c_{4}^{2}H^{2}}{c_2}\left \| e \right \|_{s},\ \forall\ e \in L^{2}{\left(\Omega \right) }
\end{equation}

\begin{proof}
Under the assumptions of the Lax-Milgram Lemma the accuracy of the real solution by the Galerkin solution  is as good as the best approximation of $\text{u} \in V$ by a function in $V$ which reduces this proof to a problem of approximation theory as following: For any $\phi \in V$, 
\begin{equation*}
    \left\{\begin{matrix}
a_{\text{DG}}\left ( F_{H}\phi ,F_{H}E_{h}S_{H}\phi  \right )=\int_\Omega f\phi \\ 
\quad a_{\text{DG}}\left ( \phi ,F_{H}E_{h}S_{H}\phi  \right )=\int_\Omega f\phi.
\end{matrix}\right.
\end{equation*}
It is based on the observation that the error $\phi -F_{H}\phi$ is a-orthogonal to $V$, i.e.,
\begin{equation}
    a\left(\phi-F_{H}\phi, F_{H}E_{h}S_{H}\phi  \right)=0, \ \forall\ \phi  \in V,
\end{equation}
that is
\begin{equation}
    a\left(S_{H}\phi, F_{H}E_{h}S_{H}\phi  \right)=0, \ \forall\  \phi  \in V,\label{qaz}
\end{equation}
a property which is referred to as Galerkin orthogonality.  By (\ref{equ:conver1}), we have
\begin{equation*}
\begin{aligned}
\left \| S_H E_h e \right \|_{\text{DG}}^2& =a_{\text{DG}}\left ( E_he, S_H E_h e \right )\\
&=s\left ( e, S_H E_h e \right )\\
&\leq \left \| e \right \|_{ s }\left \| S_H E_h e \right \|_{s }.
\end{aligned}
\end{equation*}
 Furthermore, since $K_{j}, j= 1,...,N_1 $ are disjoint, 
 there also holds
\begin{equation*}
\begin{aligned}
\left \| S_H E_h e \right \|_{s }^2&=\sum_{j=1}^{N_1}\left \| S_H E_h e \right \|_{s\left ( K_j \right ) }^2\\
&=\sum_{j=1}^{N_1}\left \| \left ( 1-\pi_j \right )S_H E_h e \right \|_{s\left ( K_j \right ) }^2\\
&\leq \frac{c_4^2 H^2}{c_2}\sum_{j=1}^{N_1}a_{\text{DG}}^{j}\left ( S_H E_he, S_H E_h e \right )\\
&=\frac{c_4^2 H^2}{c_2}\left \| S_H E_h e \right \|_{\text{DG}}^2.
\end{aligned}
\end{equation*}
This proofs  (\ref{equ:lemma2}).

On the other hand, let $\phi=E_{h}e$, 
(\ref{qaz}) implies 
\begin{equation}
\begin{aligned}
\left \| S_{H}E_{h}e \right \|_{s}^{2}
   &=a_{\text{DG}}\left ( E_{h}S_{H} E_{h}e, S_{H} E_{h}e \right )\\
   &=a_{\text{DG}}\left ( S_{H} E_{h} S_{H} E_{h}e, S_{H} E_{h}e \right )\\
   &\leq 2 \left \|S_{H} E_{h} S_{H} E_{h}e  \right \|_{\text{DG}}\left \| S_{H} E_{h}e \right \|_{\text{DG}}\\
   &\leq \frac{c_4 H}{\sqrt{c_2}}\left \| S_{H} E_{h}e \right \|_{s}\left \| S_{H} E_{h}e \right \|_{\text{DG}}.
\end{aligned}
\end{equation}
That is 
\begin{equation}
\begin{aligned}
    \left \| S_{H} E_{h}e  \right \|_{s}
    &\leq \frac{c_4 H}{\sqrt{c_2}}\left \| \left ( S_{H} E_{h}e  \right ) \right \|_{\text{DG}}\\
&\leq \frac{c_4^2 H^2}{c_2}\left \| e\right \|_{s}
\end{aligned}
\end{equation}
We finish the proof.
\end{proof}

\end{lemma}\label{lemma2}
 Now, we are going to eatimate the error between the fine-scale solution $u_{h}^{m}=u_{h}\left ( t_{m} \right )$ obtained from solving (\ref{adg}) and the coarse-scale solution $u_{H}^{m}$ obtianed from solving (\ref{adgtime}). 
 Regarding the error of the ellipse projection, we make the following estimates.
 \begin{lemma}\label{lemmaa}
 Assuming  $f\in C^{4} \left(\left[0,T\right];L^{2}\left(\Omega\right)\right),$ there exists $c_5 \geq 0$ such that 
 \begin{equation}
 \begin{aligned}
\left \| S_{H}u _{h}^{m} \right \|_{s}&\leq c_5\frac{H^{2}}{c_2}\left ( \left \| e \right \|_{C^0 \left ( \left [ 0,T \right ]; L^{2} \left ( \Omega  \right )\right )}+\tau ^{2} \left \| \partial_{t} ^{2}e \right \|_{C^0 \left ( \left [ 0,T \right ]; L^{2} \left ( \Omega  \right )\right )}\right ),\\
\left \| \frac{S_{H}u _{h}^{m+1}-S_{H}u _{h}^{m}}{\tau } \right \|_{s}&\leq c_5\frac{H^{2}}{c_2}\left (   \left \| \partial_{t}^2 e \right \|_{C^0 \left ( \left [ 0,T \right ]; L^{2} \left ( \Omega  \right )\right )}+\tau ^{2} \left \| \partial_{t} ^{4}e \right \|_{C^0 \left ( \left [ 0,T \right ]; L^{2} \left ( \Omega  \right )\right )}\right ),\\
\left \| u _{H}^{1}-F_{H}u _{h}^{1}-\left ( u _{H}^{0}-F_{H}u _{h}^{0} \right ) \right \|_{s}&\leq c_5 \tau \left (\frac{H^{2}}{c_2}\left \| \partial_{t} e \right \|_{C^{0}\left ( \left [ 0,T \right ]; L^{2} \left ( \Omega  \right )\right )}+\tau ^{2}\left \| \partial_{t} ^{3}\left(u_{h}\right)\right \|_{C^{0}\left ( \left [ 0,T \right ]; L^{2} \left ( \Omega  \right )\right )} \right ),
  \end{aligned}
 \end{equation}
where $e=f-\partial_{t} ^{2}\left(u _{h}\right).$
 \end{lemma}
 \begin{proof}
 \quad \\ \textbf{Part\ A.}\\
 By the above definitions $u _{h}^{m}=E_{h}\left ( e\left ( \cdot ,t_{m} \right ) \right )$ and Taylor's theorem, we have
\begin{equation}
e\left ( \cdot ,t_{m}+t \right )=e\left ( \cdot ,t_{m} \right )+t\partial_{t} e\left ( \cdot ,t_{m} \right )+\int_{t_{m}}^{t_{m+1}}r\partial_{t} ^{2}e\left ( \cdot ,r \right )\text{dr}.
\end{equation}
By integrating from $t=\tau $ to $t=-\tau$, then we have
\begin{equation*}
    \int_{-\tau}^{\tau}e\left ( \cdot ,t_{m}+t \right )\text{dt}=\int_{-\tau}^{\tau}\left ( e\left ( \cdot ,t_{m} \right )+t\partial_{t} e\left ( \cdot ,t_{m} \right )+\int_{t_{m}}^{t_{m+1}}r\partial_{t} ^{2}e\left ( \cdot ,r \right )\text{dr}\right )\text{dt}.
\end{equation*}
We see that
\begin{equation*}
    2\tau\int_{-\tau}^{\tau}e\left ( \cdot ,t_{m} \right )\text{dt}.=\int_{-\tau}^{\tau} e\left ( \cdot ,t_{0} \right )\text{dt}.+0+\int_{-\tau}^{\tau}\tau\int_{t_{m}-\tau}^{t_{m}+\tau}\partial_{t} ^{2}e\left ( \cdot ,r \right )\text{drdt},\ t_{0}\in \left(t_{m-1},t_{m+1}\right).
\end{equation*}
Then we have
\begin{equation}
     \left \| e\left ( \cdot ,t_{m} \right ) \right \|_{s}\leq \frac{1}{2\tau }\left \| e \right \|_{L^{1}\left ( t_{m-1},t_{m+1};L^{2}\left ( \Omega  \right ) \right )}+\frac{\tau }{2}\left \| \partial_{t} ^{2}e \right \|_{L^{1}\left ( t_{m-1},t_{m+1};L^{2}\left ( \Omega  \right ) \right )}.\label{5}
\end{equation}
Hence we finish the proof of the first inequality.\\
 \textbf{Part\ B.}\\
Taking $t=\tau $ and $t=-\tau $ repectively, we also have
\begin{equation*}
\begin{aligned}
    e\left ( \cdot ,t_{m}+\tau \right )&=e\left ( \cdot ,t_{m}\right )+t\partial_{t} e\left ( \cdot ,t_{m} \right )+\frac{t^{2}}{2}\partial_{t} ^{2}e\left ( \cdot ,t_{m}\right)+\frac{t^{3}}{3}\partial_{t} ^{3}e\left ( \cdot ,t_{m}\right)+\int_{t_{m}}^{t_{m}+t}\frac{t^{3}}{3}\partial_{t} ^{4}e\left ( \cdot ,r\right )\text{dr},\\
    e\left ( \cdot ,t_{m}-\tau \right )&=e\left ( \cdot ,t_{m}\right )-t\partial_{t} e\left ( \cdot ,t_{m} \right )+\frac{t^{2}}{2}\partial_{t} ^{2}e\left ( \cdot ,t_{m}\right)-\frac{t^{3}}{3}\partial_{t} ^{3}e\left ( \cdot ,t_{m}\right)+\int_{t_{m}}^{t_{m}+t}\frac{t^{3}}{3}\partial_{t} ^{4}e\left ( \cdot ,r\right )\text{dr}.
    \end{aligned}
\end{equation*}
Then we make a difference that 
\begin{equation}
    \begin{aligned}
   \left \| \frac{S_{H}u _{h}^{m+1}-S_{H}u _{h}^{m}}{\tau } \right \|_{s}  &\leq
   \left \| e\left ( \cdot ,t_{m+1} \right )- 2e\left ( \cdot ,t_{m} \right )+e\left ( \cdot ,t_{m-1} \right )\right \|_{s}\\
   &\leq \tau \left \| \partial_{t}^{2} e\left ( \cdot ,t_{m} \right ) \right \|_{s}+\frac{\tau ^{2}}{2}\left \| \partial_{t} ^{4}e \right \|_{L^{1}\left ( t_{m},t_{m+1};L^{2}\left ( \Omega  \right ) \right )}.
    \end{aligned}
\end{equation}
It is easily to finish the proof of the second result.\\
 \textbf{Part\ C .}\\
For the third result, we obtian
\begin{equation}
    \left \| e\left ( \cdot ,\tau  \right )-e\left ( \cdot ,0 \right ) \right \|_{s}\leq \tau \left \|  \partial_{t} g\right \|_{C\left ( \left [ 0,T \right ];L^{2}\left ( \Omega  \right ) \right )},
\end{equation}
then there exists a constant $c_5>0,$ such that 
\begin{equation}
    \left \| S_{H}u _{h}^{1}- S_{H}u _{h}^{0}\right \|_{s}\leq \frac{c_5 H^{2}}{c_2}\tau \left \| \partial_{t} f- \partial_{t} ^{3}\left(u _{h}\right)\right \|_{C^{0}\left ( \left [ 0,T \right ];L^{2}\left ( \Omega  \right ) \right )}.
\end{equation}
Using Taylor's theorem on $u _{h}$, we have
\begin{equation}
    u _{h}\left ( \cdot ,\tau  \right )=u _{h}\left ( \cdot ,0 \right )+\tau \partial_{t} \left(u _{h}\right)\left ( \cdot ,0 \right )+\frac{\tau^{2}}{2} \partial_{t}^{2} \left(u _{h}\right)\left ( \cdot ,t_{m} \right )+\int_{0}^{\tau}\frac{r^{2}}{2}\partial_{t}^{3} \left(u _{h}\right)\left ( \cdot ,r \right )\text{dr}.
\end{equation}
By the definition of $ u _{H}^{0},u _{H}^{1}\in V_{\rm{cem}}^r $ and taking any $u \in V_{\rm{cem}}^{r}$, we have
\begin{equation}
s\left(u_{H}^{1}, v\right)=s\left(u_{h}^{1}, v\right)=s\left(u_{h}^{0},v\right)+\tau s\left(\partial_{t} \left(u^0 _{h}\right),v\right)+\frac{\tau^{2}}{2}s\left( \partial_{t}^{2} \left( \widetilde{u}^0_{h}\right),v\right).
\end{equation}
Thus,
\begin{equation}
\begin{aligned}
s\left(u_{H}^{1}-u_{h}^{1}-\left ( u _{h}^{0}-u _{H}^{0} \right ), v\right) 
&=\frac{\tau ^{2}}{2}s\left(\partial_t \left (\widetilde{u}^0_{h}  \right )-\partial_t^2\left(u_{h}^{0}\right), v\right)-\int_{0}^{\tau }\frac{r^{2}}{2}s\left ( \partial_{t}^{3} \left(u _{h}\right)\left ( \cdot ,r \right ),v\right )\text{dr} \\
&=\frac{\tau ^{2}}{2}\left[s\left(f^{0},v\right)-a_{\text{DG}}\left(u_{h}^{0},v\right)+s\left(\partial_{t}^{2}\left(u_{h}^{0}\right), v\right)\right]-\int_{0}^{\tau}\frac{r^{2}}{2}s\left ( \partial_{t}^{3} \left(u _{h}\right)\left ( \cdot ,r \right ),v\right )\text{dr}\\
&=-\int_{0}^{\tau }\frac{r^{2}}{2}s\left ( \partial_{t}^{3} \left(u _{h}\right)\left ( \cdot ,r \right ),v \right )\text{dr}
\end{aligned}
\end{equation}
This yields 
\begin{equation}
    \left \| u _{h}^{1}-u _{H}^{1}-\left ( u _{h}^{0}-u _{H}^{0} \right ) \right \|_{s}\leq \frac{\tau ^{3}}{6}\left \| \partial_{t}^{3} \left(u _{h}\right) \right \|_{C^{0}\left ( \left [ 0,T \right ];L^{2}\left ( \Omega  \right ) \right )}.
\end{equation}
Finally, we finish the proof of the third inequality.
 \end{proof}
\begin{theorem}\label{thm1}
Assuming $f\in C^4 \left ( \left [ 0,T \right ];H^{1} \left ( \Omega  \right )\right )$, the fine-scale solution $u_{h}^{m}=u_{h}\left ( t_{m} \right )$ obtained from solving (\ref{adg}) and the coarse-scale solution $u_{H}^{m}$ obtianed from solving (\ref{adgtime}), we have the following estimate
\begin{equation}
\begin{aligned}
   &\quad \left \| \pi\left ( \frac{u _{H}^{m+1}-F_{H}u _{h}^{m+1}-\left ( u _{H}^{m}-F_{H}u _{h}^{m} \right )}{\tau } \right ) \right\|_{s}^{2}+a_{\text{DG}}\left (u _{H}^{m}-F_{H}u _{h}^{m}, u _{H}^{m}-F_{H}u _{h}^{m}\right )\\
   &\leq 12c_5 \left(\tau+4\tau ^{2}\left(1- \eta \right)^{-1}+\frac{H^{2}}{c_2}c_5\right).
\end{aligned}
\end{equation}
\end{theorem}

\begin{proof}
From above lemma we note that $a_{\text{DG}}\left(S_{H}u_{h}^{m},v\right)=0.$  Then we have
\begin{equation*}
    \begin{aligned}
       &\quad {\pi}'\left ( \frac{u _{H}^{m+1}-F_{H}u _{h}^{m+1}-2\left ( u _{H}^{m}-F_{H}u _{h}^{m} \right )+u _{H}^{m-1}-F_{H}u _{h}^{m-1}}{\tau^{2} },v \right )
       +a_{\text{DG}}\left ( u _{H} ^{m}-F_{H}u _{h}^{m},v\right )\\
       &={\pi}'\left ( \frac{u _{H}^{m+1}-2u _{H}^{m}+u _{H}^{m-1}-F_{H}\left ( u _{h}^{m+1}-2u _{h}^{m}+u _{h}^{m-1}\right )}{\tau^{2}}, v\right)
       +a_{\text{DG}}\left ( u _{H} ^{m}-F_{H}u _{h}^{m},v\right )\\
       &={\pi}'\left ( \frac{u _{H}^{m+1}-2u _{H}^{m}+u _{H}^{m-1}-\left ( u _{h}^{m+1}-2u _{h}^{m}+u _{h}^{m-1} \right )+\left ( u _{h}^{m+1}-2u _{h}^{m}+u _{h}^{m-1} \right )-F_{H}\left ( u _{h}^{m+1}-2u _{h}^{m}+u _{h}^{m-1}\right )}{\tau^{2}}, v\right)\\
       &\quad +a_{\text{DG}}\left ( u _{H} ^{m}-u _{h} ^{m}+u _{h} ^{m}-F_{H}u _{h}^{m},v\right)\\
       &={\pi}'\left ( \frac{u _{H}^{m+1}-2u _{H}^{m}+u _{H}^{m-1}-\left ( u _{h}^{m+1}-2u _{h}^{m}+u _{h}^{m-1} \right )+\left ( S_{H}u _{h}^{m+1}-2S_{H}u _{h}^{m}+S_{H}u _{h}^{m-1}\right )}{\tau^{2}}, v\right)\\
       &\quad+a_{\text{DG}}\left ( u _{H} ^{m}-u _{h} ^{m},v\right)+a_{\text{DG}}\left(S_{H}u_{h}^{m},v\right)\\
       &={\pi}'\left ( \frac{u _{H}^{m+1}-2u _{H}^{m}+u _{H}^{m-1}}{\tau^{2}}, v\right)-{\pi}'\left ( \frac{ u _{h}^{m+1}-2u _{h}^{m}+u _{h}^{m-1}}{\tau^{2}},v \right )+{\pi}'\left ( \frac{ S_{H}u _{h}^{m+1}-2S_{H}u _{h}^{m}+S_{H}u _{h}^{m-1}}{\tau^{2}},v \right )\\
       &\quad+a_{\text{DG}}\left ( u _{H} ^{m},v\right)
       -a_{\text{DG}}\left(u _{h} ^{m},v\right),\label{theoremone}
    \end{aligned}
\end{equation*}
for any $v\in V_{\rm{cem}}^r.$

Using (\ref{adg}) and (\ref{adgtime}), we have
\begin{equation*}
    \begin{aligned}
       &\quad {\pi}'\left ( \frac{u _{H}^{m+1}-F_{H}u _{h}^{m+1}-2\left ( u _{H}^{m}-F_{H}u _{h}^{m} \right )+u _{H}^{m-1}-F_{H}u _{h}^{m-1}}{\tau^{2} },v \right )
       +a_{\text{DG}}\left ( u _{H} ^{m}-F_{H}u _{h}^{m},v\right )\\
       &=\left[\pi\left ( \frac{u _{H}^{m+1}-2u _{H}^{m}+u _{H}^{m-1}}{\tau^{2}}, v\right)+a_{\text{DG}}\left ( u _{H} ^{m},v\right)\right]-\pi\left (\frac{ u _{h}^{m+1}-2u _{h}^{m}+u _{h}^{m-1}}{\tau^{2}},v \right )\\
       &\quad+\pi\left ( \frac{ S_{H}u _{h}^{m+1}-2S_{H}u _{h}^{m}+S_{H}u _{h}^{m-1}}{\tau^{2}} ,v\right )+\left [ \pi \left ( \partial_{t}^{2} \left ( u _{h} \right ),v\right )-\left ( f^{m},v\right ) \right ]\\
       &=\pi\left ( \frac{ S_{H}u _{h}^{m+1}-2S_{H}u _{h}^{m}+S_{H}u _{h}^{m-1}}{\tau^{2}} ,v\right )+ \left[\pi \left ( \partial_{t}^{2} \left ( u _{h} \right ),v\right )-\pi\left (\frac{ u _{h}^{m+1}-2u _{h}^{m}+u _{h}^{m-1}}{\tau^{2}},v \right )\right]\\
       &\quad +\left[\pi\left ( f^{m},v\right )-\left ( f^{m},v\right )\right]. 
    \end{aligned}
\end{equation*}
This simplies that
\begin{equation}
    \begin{aligned}
       &\quad {\pi}'\left ( \frac{u _{H}^{m+1}-F_{H}u _{h}^{m+1}-2\left ( u _{H}^{m}-F_{H}u _{h}^{m} \right )+u _{H}^{m-1}-F_{H}u _{h}^{m-1}}{\tau^{2} },v\right )
       +a_{\text{DG}}\left ( u _{H} ^{m}-F_{H}u _{h}^{m},v\right )\\
       &=\pi\left ( \frac{ S_{H}u _{h}^{m+1}-2S_{H}u _{h}^{m}+S_{H}u _{h}^{m-1}}{\tau^{2}} ,v\right )+ \left[\pi \left ( \partial_{t}^{2} \left ( u _{h} \right )-\frac{ u _{h}^{m+1}-2u _{h}^{m}+u _{h}^{m-1}}{\tau^{2}},v\right )\right]
       +\left(\pi-I\right)\left ( f^{m},v\right ).
    \end{aligned}
\end{equation}
    Using Lemma \ref{lemma1}, we note that
  \begin{equation}
    \begin{aligned}
      & \quad \left \| \pi\left ( \frac{u _{H}^{m+1}-F_{H}u _{h}^{m+1}-\left ( u _{H}^{m}-F_{H}u _{h}^{m} \right )}{\tau } \right ) \right\|_{L^{2}\left(\Omega\right)}^{2}
      +a_{\text{DG}}\left (u _{H}^{m}-F_{H}u _{h}^{m}, u _{H}^{m}-F_{H}u _{h}^{m}\right )\\
      &\leq \frac{16}{15}\left (\underset{\mathbf{Part.1.}}{\underbrace{Q^{1} \left ( u_{H}-F_{H}u_{h}\right )}}
      +\underset{\mathbf{Part.2.}}{\underbrace{4\tau ^{2}\left(1- \frac{c_{1}\tau ^{2}}{4c_{3}H^2} \right)^{-1}\left(G^{m}\right)^{2}}}
      +\underset{\mathbf{Part.3.}}{\underbrace{\frac{2H^{2}}{c_2}\left ( D^{m}\right )^{2}}}\right).
    \end{aligned}
\end{equation}
where 
\begin{equation*}
\begin{aligned}
     G^{m}&=\sum_{p=1}^{m}\left \| \pi\left ( \frac{ S_{H}u _{h}^{m+1}-2S_{H}u _{h}^{m}+S_{H}u _{h}^{m-1}}{\tau^{2}}{\tau } \right ) \right \|_{s}\\
     &\quad +\sum_{p=1}^{m}\left \| \pi\left (\partial_{t} ^{2}\left(u _{h}\right)\left ( t_{p},\cdot  \right )-\frac{u _{h}^{m+1}-2u _{h}^{m}+u _{h}^{m-1}}{\tau^{2}}\right ) \right \|_{s},\\
     D^{m}&=2\left(\sum_{p=1,n}\left \|  Pf^{p} \right \|_{s}+\tau \sum_{k=1}^{m}\left \| P\left ( \frac{f ^{k+1}-f ^{k}}{\tau } \right ) \right \|_{s}\right).
\end{aligned}
\end{equation*}
\begin{enumerate}
\item [\textbf{Part\ 1.}]
 Considering $m=0$ in Lemma \ref{lemma1}, we note that
\begin{equation}
\begin{aligned}
     Q^{1}\left ( u_{H}-F_{H}u_{h}  \right )
     &\leq \left \| \pi\left ( \frac{u_{H}^{1}-F_{H}u_{h}^{1} -\left ( u_{H}^{0}-F_{H}u_{h}^{0} \right )}{\tau} \right ) \right \|_{s}\\
      &\quad -\left ( \frac{\tau}{2} \right )^{2}a_{\text{DG}}\left ( \frac{u_{H}^{1}-F_{H}u_{h}^{1} -\left ( u_{H}^{0}-F_{H}u_{h}^{0} \right )}{\tau},\frac{u_{H}^{1}-F_{H}u_{h}^{1} -\left ( u_{H}^{0}-F_{H}u_{h}^{0} \right )}{\tau}\right )\\
      &\quad +a_{\text{DG}}\left ( u_{H}^{0}-F_{H}u_{h}^{0},u_{H}^{0}-F_{H}u_{h}^{0} \right )\\
     &\leq\left ( 1+\frac{c_{1}\tau ^{2}}{4c_{3}H^2}    \right )\left \| \pi\left ( \frac{u_{H}^{1}-F_{H}u_{h}^{1} -\left ( u_{H}^{0}-F_{H}u_{h}^{0} \right )}{\tau} \right ) \right \|_{s}+ \left \| u_{H}^{0}-F_{H}u_{h}^{0}\right \|_{\text{DG}}^{2}.
\end{aligned}
\end{equation}
By Lemma \ref{lemma2} and Lemma \ref{lemmaa}, we have 
\begin{equation}
      Q^{1}\left ( u_{H}-F_{H}u_{h}  \right )
      \leq \left ( 1+\frac{c_{1}\tau ^{2}}{4c_{3}H^2} \right )\left ( c_{5}\tau \left ( \frac{H^{2}}{c_2}+\tau ^{2} \right ) \right ).
\end{equation}
\item [\textbf{Part\ 2.}]
Next, with the second inequality in Lemma \ref{lemmaa}, $G^{m}$ can be estimated by
\begin{equation}
    \sum_{p=1}^{m}\left \| \pi\left ( \frac{ S_{H}u _{h}^{m+1}-2S_{H}u _{h}^{m}+S_{H}u _{h}^{m-1}}{\tau^{2}}{\tau } \right ) \right \|_{s}
    \leq \left ( 1+\frac{c_2}{H^{2}} \right )\frac{c_5 H^{2}}{\tau c_2}.
\end{equation}
Similarly, using Taylor's expansion, then we estimate the second term in $G^{m}$
\begin{equation}
\begin{aligned}
    &\quad \sum_{p=1}^{m}\left \| \pi\left (\partial_{t}  ^{2}\left(u _{h}\right)\left ( t_{p},\cdot  \right )-\frac{ u _{h}^{m+1}-2u _{h}^{m}+u _{h}^{m-1}}{\tau^{2}}\right ) \right \|_{s}\\
    &\leq \left ( 1-\frac{c_2}{H^{2}} \right )\tau \left \| \partial_{t}  ^{4}\left(u _{h}\right)\right \|_{C\left ( \left [ 0,T \right ];L^{2}\left ( \Omega  \right ) \right )}.
\end{aligned}
\end{equation}
Note that the Part 2. as following 
\begin{equation}
    \begin{aligned}
  4\tau ^{2}\left(1- \frac{c_{1}\tau ^{2}}{4c_{3}H^2} \right)^{-1}\left(G^{m}\right)^{2}&\leq 4\tau ^{2}\left(1- \frac{c_{1}\tau ^{2}}{4c_{3}H^2} \right)^{-1}\left[ \left ( 1+\frac{c_2}{H^{2}} \right )\frac{c_5 H^{2}}{\tau c_2}+\left ( 1-\frac{c_2}{H^{2}} \right )\tau \left \| \partial_{t}  ^{4}\left(u _{h}\right) \right \|_{C\left ( \left [ 0,T \right ];L^{2}\left ( \Omega  \right ) \right )}\right]^{2}\\
    &\leq 4\tau ^{2}\left(1- \frac{c_{1}\tau ^{2}}{4c_{3}H^2} \right)^{-1} \left[\frac{c_1}{2}\left(\frac{H^{2}}{c_2}+1\right)\right]^{2}.\label{equ:gm}
    \end{aligned}
\end{equation}
\item [\textbf{Part\ 3.}]
By Lemma \ref{lemma2}, we estimate the terms in $D^{m}$ by
\begin{equation}
    \begin{aligned}
    \left \| P\left ( f^{p} \right ) \right \|_{s}
    &\leq \frac{H}{\sqrt{2c_2}}\left \| f \right \|_{C\left ( \left [ 0,T \right ];H^{1} \left ( \Omega  \right )\right )}, \  \forall \ 1\leq p\leq n, \\
     \tau \sum_{p=1}^{m}\left \| P\left ( \frac{f^{p+1}-f^{p}}{\tau } \right ) \right \|_{s}
     &\leq c_5 \left \| \partial f_{t} \right \|_{C\left ( \left [ 0,T \right ];L^{2}\left ( \Omega  \right )\right )}.
 \end{aligned}
\end{equation}
Note that the Part 3. as following 
\begin{equation}
    \begin{aligned}
    \frac{2H^{2}}{c_2}\left ( D^{m}\right )^{2}
    &\leq \frac{2H^{2}}{c_2}\left(2\times\frac{H}{\sqrt{2c_2}}\left \| f \right \|_{C\left ( \left [ 0,T \right ];H^{1} \left ( \Omega  \right )\right )}+ c_5 \left \| \partial f_{t} \right \|_{C\left ( \left [ 0,T \right ] ;L^{2}\left ( \Omega  \right )\right )}\right)\\
    &\leq \frac{2H^{2}}{c_2}\left(\frac{\sqrt{2}H}{\sqrt{c_2}}+c_5 \right)^{2}.\label{equ:dm}
    \end{aligned}
\end{equation}
\end{enumerate}

 Combining all these estimates, we obtain
 \begin{equation}
 \begin{aligned}
      &\quad \frac{16}{15}\left(Q^{1}\left(u_{H}-F_{H}u_{h}\right )+4\tau ^{2}\left(1- \frac{c_{1}\tau ^{2}}{4c_{3}H^2} \right)^{-1}\left(G^{m}\right)^{2}+\frac{2H^{2}}{c_2}\left (D^{m}\right)^{2}\right)\\
      &\leq \frac{16}{15}\left(6c_5\tau+4\tau ^{2}\left(1- \frac{c_{1}\tau ^{2}}{4c_{3}H^2} \right)^{-1}\left(2c_5^{2}+1\right)+\frac{8H^{2}}{c_2}c_{5}^{2}\right)\\
      &\leq 12c_5 \left(\tau+4\tau ^{2}\left(1- \frac{c_{1}\tau ^{2}}{4c_{3}H^2} \right)^{-1}+\frac{H^{2}}{c_2}c_5\right).
 \end{aligned}
 \end{equation}
Finally we finish the proof.
\end{proof}

\begin{theorem}\label{theorem3}
The fine-scale solution $u_{h}^{m}=u_{h}\left ( t_{m} \right )$ obtained from solving (\ref{adg}) and the coarse-scale solution $u_{H}^{m}$ obtianed from solving (\ref{adgtime}), we have the following error estimate
\begin{equation}
    \underset{0\leq p\leq N_1-1}{\rm{max}}\left \|u _{h}^{p}- u _{H}^{p}\right \|_{s}
    \leq 12c_5 \left(\tau+4\tau ^{2}\left(1- \eta \right)^{-1}+\frac{H^{2}}{c_2}c_5\right).
\end{equation}
\end{theorem}

\begin{proof}
Taking sum of (\ref{theoremone}), we have
\begin{equation}
    \begin{aligned}
       & \quad \sum_{p=1}^{p} {\pi}'\left ( \frac{u _{H}^{p+1}-F_{H}u _{h}^{p+1}-\left ( u _{H}^{p}-F_{H}u _{h}^{p} \right )}{\tau },v \right ) +\sum_{p=1}^{m} a_{\text{DG}}\left ( u _{H} ^{p}-F_{H}u _{h}^{p},v\right )\\
       &=\sum_{p=1}^{m} s\left ( \pi \left ( \frac{S_{H}u _{h}^{p+1}-S_{H}u _{h}^{p}}{\tau } \right )+\pi\left ( \partial_t ^{2}\left(u _{h}\right)\left ( t_{p},\cdot  \right ) -\frac{u _{h}^{p+1}-u _{h}^{p}}{\tau }\right )+P\left ( f^{p} \right ), v \right ).
    \end{aligned}
\end{equation}
Moreover, we imply
\begin{equation}
    \begin{aligned}
    &\quad \left \| \pi \left ( u _{H}^{m+1}-F_{H}u _{h}^{m+1} \right ) \right \|_{s}^{2}+a_{\text{DG}}\left ( \tau \sum_{p=1}^{m}\left ( u _{H}^{p}-F_{H}u _{h}^{p} \right ), \tau \sum_{p=1}^{m+1}\left ( u _{H}^{p}-F_{H}u _{h}^{p} \right )\right )\\
   &={\pi}'\left ( \left ( u _{1}^{H}-F_{H}u _{h}^{1}-\left (u _{0}^{H}-F_{H}u _{h}^{0}  \right ) \right ) ,u _{m+1}^{H}-F_{H}u _{h}^{m+1}\right )\\
   &\quad +\tau ^{2}s\left ( \pi \left ( \frac{S_{H}u _{h}^{m+1}-S_{H}u _{h}^{m}}{\tau } \right )+\pi\left ( \partial_t ^{2}\left(u _{h}\right)-\frac{u _{h}^{m+1}-u _{h}^{m}}{\tau }\right )+P\left ( f^{m} \right ),   u _{H}^{m+1}-u _{H}^{m} \right ).\label{22}
    \end{aligned}
\end{equation}
Replacing $m$ as $m+1$ in the (\ref{22}) and substracting these two equations, we note that
\begin{equation}
    \begin{aligned}
    &\quad \quad\left \| \pi \left ( u _{H}^{m+1}-F_{H}u _{h}^{m+1} \right ) \right \|_{s}^{2}- \left \| \pi \left ( u _{H}^{m}-F_{H}u _{h}^{m} \right ) \right \|_{s}^{2}\\
    &\quad +a_{\text{DG}}\left ( \tau \sum_{p=1}^{m}\left ( u _{H}^{p}-F_{H}u _{h}^{p} \right ), \tau \sum_{p=1}^{m+1}\left ( u _{p}^{H}-F_{H}u _{h}^{p} \right )\right )\\
    &\quad -a_{\text{DG}}\left ( \tau \sum_{p=1}^{m-1}\left ( u _{H}^{p}-F_{H}u _{h}^{p} \right ), \tau \sum_{p=1}^{m}\left ( u _{H}^{p}-F_{H}u _{h}^{p} \right )\right )\\
   &={\pi}'\left ( \left ( u _{H}^{1}-F_{H}u _{h}^{1}-\left (u _{H}^{0}-F_{H}u _{h}^{0}  \right ) \right ) ,u _{H}^{m+1}-F_{H}u _{h}^{m+1}-\left(u _{H}^{m}-F_{H}u _{h}^{m}\right)\right )\\
   &\quad +\sum_{p=1}^{m+1}\tau ^{2}s\left ( {\pi \left ( \frac{S_{H}u _{h}^{m+1}-S_{H}u _{h}^{m}}{\tau } \right )+\pi\left ( \partial_t ^{2}\left(u _{h}\right)\left ( t_{m},\cdot  \right ) -\frac{u _{h}^{m+1}-u _{h}^{m}}{\tau }\right )+P\left ( f^{m} \right ),} \right.\\
  &\qquad \left. { u _{H}^{m+1}-F_{H}u _{h}^{m+1}+u _{H}^{m}-F_{H}u _{h}^{m}} \right).
\end{aligned}
\end{equation}
Again using telescoping sum, we have
\begin{equation}
    \begin{aligned}
    &\quad \quad\left \| \pi \left ( u _{H}^{m+1}-F_{H}u _{h}^{m+1} \right ) \right \|_{s}^{2}-\left \| \pi \left ( u _{H}^{1}-F_{H}u _{h}^{1} \right ) \right \|_{s}^{2}\\
    &\quad +a_{\text{DG}}\left ( \tau \sum_{p=1}^{m}\left ( u _{H}^{p}-F_{H}u _{h}^{p} \right ), \tau \sum_{p=1}^{m+1}\left ( u _{H}^{p}-F_{H}u _{h}^{p} \right )\right )\\
   &=\sum_{p=1}^{m}s\left ( {\pi}'\left ( u _{H}^{1}-F_{H}u _{h}^{1}-\left ( u _{H}^{0}-F_{H}u _{h}^{0} \right ) \right ), u _{H}^{p+1}-F_{H}u _{h}^{p+1}-\left ( u _{H}^{p}-F_{H}u _{h}^{p} \right ) \right )\\
   &\quad +\sum_{p=1}^{m}\sum_{r=1}^{p}\tau ^{2}s\left ( {\pi \left ( \frac{S_{H}u _{h}^{r+1}-S_{H}u _{h}^{r}}{\tau } \right )+\pi\left (  \partial_t ^{2}\left(u _{h}\right)-\frac{u _{h}^{r+1}-u _{h}^{r}}{\tau }\right )+P\left ( f^{r} \right ),} \right.\\
   &\qquad \left. { u _{H}^{r+1}-F_{H}u _{h}^{r+1}-u _{H}^{r}+F_{H}u _{h}^{r}} \right).\label{23}\\
    \end{aligned}
\end{equation}
Next, we estimate the error of (\ref{23}). For the third term on left hand side of  (\ref{23}), we obtain
\begin{equation}
    \begin{aligned}
    &\qquad a_{\text{DG}}\left ( \tau \sum_{p=1}^{m}\left ( u _{H}^{p}-F_{H}u _{h}^{p} \right ), \tau \sum_{p=1}^{m+1}\left ( u _{H}^{p}-F_{H}u _{h}^{p} \right )\right )\\
    &\quad =\frac{1}{4}a_{\text{DG}}\left ( \tau \sum_{p=1}^{m}\left ( u _{H}^{p}-F_{H}u _{h}^{p} \right )+\tau \sum_{p=1}^{m+1}\left ( u _{H}^{p}-F_{H}u _{h}^{p} \right ),\tau \sum_{p=1}^{m}\left ( u _{H}^{p}-F_{H}u _{h}^{p} \right )+\tau \sum_{p=1}^{m+1}\left ( u _{H}^{p}-F_{H}u _{h}^{p} \right )\right )\\
   &\quad \quad -a_{\text{DG}}\left (\tau \sum_{p=1}^{m+1}\left ( u _{H}^{p}-F_{H}u _{h}^{p} \right )-\tau \sum_{p=1}^{m+1}\left ( u _{H}^{p}-F_{H}u _{h}^{p} \right ),\tau \sum_{p=1}^{m}\left ( u _{H}^{p}-F_{H}u _{h}^{p} \right ) -\tau \sum_{p=1}^{m}\left ( u _{H}^{p}-F_{H}u _{h}^{p} \right ) \right )\\
   &\quad =a_{\text{DG}}\left ( \frac{\tau \sum_{p=1}^{m}\left ( u _{H}^{p}-F_{H}u _{h}^{p} \right )+\tau \sum_{p=1}^{m+1}\left ( u _{H}^{p}-F_{H}u _{h}^{p} \right )}{2} ,\frac{\tau \sum_{p=1}^{m}\left ( u _{H}^{p}-F_{H}u _{h}^{p} \right )+\tau \sum_{p=1}^{m+1}\left ( u _{H}^{p}-F_{H}u _{h}^{p} \right )}{2}\right )\\
   &\quad \quad -\frac{\tau ^{2}}{4}a_{\text{DG}}\left ( u _{H}^{m+1}-F_{H}u _{h}^{m+1} , u _{H}^{m+1}-F_{H}u _{h}^{m+1}\right )\\
   &\quad \geq \frac{1}{2}\left \| \frac{\tau \sum_{p=1}^{m}\left ( u _{H}^{p}-F_{H}u _{h}^{p} \right )+\tau \sum_{p=1}^{m+1}\left ( u _{H}^{p}-F_{H}u _{h}^{p} \right )}{2}  \right \|_{\text{DG}}^{2}-\frac{c_{1}\tau ^{2}}{4c_{3}H^2}\left \| \pi\left ( u _{H}^{m+1}-F_{H}u _{h}^{m+1} \right ) \right \|_{s}^{2}.
    \end{aligned}
\end{equation}

For the second term on the left hand side of (\ref{23}), we proceed with the standard procedure with Cauchy-Schwardz inequality. Then we possess
\begin{equation}
    \begin{aligned}
    \left \| \pi\left ( u _{h}^{1}-u _{H}^{1} \right ) \right \|_{s}^{2}
    &=\left \| \pi\left ( u _{h}^{0}-u _{H}^{0} \right ) \right \|_{s}^{2}+{\pi}'\left ( u _{h}^{1}-u _{H}^{1}-\left ( u _{h}^{0}-u _{H}^{0} \right ),u _{h}^{1}-u _{H}^{1} \right )\\
   &\quad +{\pi}'\left ( u _{h}^{1}-u _{H}^{1}-\left ( u _{h}^{0}-u _{H}^{0} \right ),u _{h}^{0}-u _{H}^{0} \right )\\
   &\leq \left \| \pi\left ( u _{h}^{0}-u _{H}^{0} \right ) \right \|_{s}^{2}+\left \| {\pi}'\left ( u _{h}^{1}-u _{H}^{1}-\left ( u _{h}^{0}-u _{H}^{0} \right ) \right ) \right \|_{s}\\
    &\quad \times \left ( \left \| \pi\left ( u _{h}^{1}-u _{H}^{1} \right ) \right \|_{L^{2}
 \left ( \Omega  \right )}^{2}+\left \| \pi\left ( u _{h}^{0}-u _{H}^{0} \right ) \right \|_{s}^{2} \right )\\
   &\leq \left \| u _{h}^{0}-u _{H}^{0} \right \|_{L^{2}}^{2}+2\left \| u _{h}^{0}-u _{H}^{0} \right \|_{L^{2}}^{2}\underset{0\leq p\leq N_1}{\max}\left \| \pi\left ( u _{h}^{p}-u _{H}^{p}  \right ) \right \|_{s}.
    \end{aligned}
\end{equation}
Similarly, for the first part on the right hand of (\ref{23}), we have
\begin{equation}
    \begin{aligned}
    &\quad {\pi}'\left ( u _{h}^{1}-u _{H}^{1}-\left ( u _{h}^{0}-u _{H}^{0} \right ), u _{h}^{p+1}-u _{H}^{p+1}-\left ( u _{h}^{p}-u _{H}^{p} \right ) \right )\\
    &\leq  2\left \|  u _{h}^{1}-u _{H}^{1}-\left ( u _{h}^{0}-u _{H}^{0} \right ) \right \|_{s}\underset{0\leq p\leq N_1}{\rm{max}}\left \| \pi\left ( u _{h}^{p}-u _{H}^{p} \right ) \right \|_{s}.
    \end{aligned}
\end{equation}
Finally, for the second part on the right hand of (\ref{23}), we note that
\begin{equation}
    \begin{aligned}
    & \quad \sum_{p=1}^{m}\sum_{r=1}^{p}\tau ^{2}s\left ( {\pi \left ( \frac{S_{H}u _{h}^{r+1}-S_{H}u _{h}^{r}}{\tau } \right )+\pi\left ( \partial_t ^{2}\left(u _{h}\right)-\frac{u _{h}^{r+1}-u _{h}^{r}}{\tau }\right )+P\left ( f^{r} \right ),} \right.\\
    &\quad \quad  \left. { u _{H}^{r+1}-F_{H}u _{h}^{r+1}+u _{H}^{r}-F_{H}u _{h}^{r}} \right)\\
    &\leq \sum_{p=1}^{m}G^{p}\left \| \pi\left ( u _{H}^{p+1}-F_{H}u _{h}^{p+1}+u _{H}^{p}-F_{H}u _{h}^{p} \right ) \right \|_{s}\\
    & \quad +\sum_{p=1}^{m}D^{p}\left \| P\left ( u _{H}^{p+1}-F_{H}u _{h}^{p+1}+u _{H}^{p}-F_{H}u _{h}^{p} \right ) \right \|_{s}\\
    &\leq 2\left ( \sum_{p=1}^{m}G^{p} \right )\underset{0\leq p \leq N_1}{\rm{max}}\left \| \pi\left ( u _{H}^{p}-F_{H}u _{h}^{p} \right ) \right \|_{s}\\
    & \quad +2\left ( \sum_{p=1}^{m}D^{p} \right )\underset{0\leq p \leq N_1-1}{\rm{max}}\left \| P\left (\frac{ u _{H}^{p+1}-F_{H}u _{h}^{p+1}+u _{H}^{p}-F_{H}u _{h}^{p}}{2}\right ) \right \|_{s}.
    \end{aligned}
\end{equation}
Using Young's inequality, we have
\begin{equation}
    \begin{aligned}
    &\quad \underset{0\leq p \leq N_1}{\rm{max}}\left \| \pi\left ( u _{H}^{p}-F_{H}u _{h}^{p} \right ) \right \|_{s}\\
    & \leq 2\left(1- \frac{c_{1}\tau ^{2}}{4c_{3}H^2} \right)^{-1} \left \|u _{H}^{0}-F_{H}u _{h}^{0} \right \|_{s}^{2}\\
    &\quad +8\left(1- \frac{c_{1}\tau ^{2}}{4c_{3}H^2} \right)^{-1}\left ( N_1\left \| u _{H}^{1}-F_{H}u _{h}^{1}-\left ( u _{H}^{0}-F_{H}u _{h}^{0} \right ) \right \|_{s}+\tau ^{2}\sum_{p=1}^{N_1-1}G^{p} \right )^{2}\\
    & \quad +4\left(1- \frac{c_{1}\tau ^{2}}{4c_{3}H^2} \right)^{-1}\left ( \tau ^{2}\sum_{p=1}^{N_1-1} D^{p}\right )^{2}+\underset{0\leq p \leq N_1-1}{\rm{max}}\left\|P\left(\frac{ u _{H}^{p+1}-F_{H}u _{h}^{p+1}+u _{H}^{p}-F_{H}u _{h}^{p}}{2}\right ) \right \|_{s}.
    \end{aligned}
\end{equation}
We can infer from (\ref{23}) that
\begin{equation}
    \begin{aligned}
    &\qquad\left(1- \frac{c_{1}\tau ^{2}}{4c_{3}H^2} \right)\left \| \pi\left ( u _{H}^{m+1}-F_{H}u _{h}^{m+1} \right ) \right \|_{s}^{2}\\
    &\quad +\frac{1}{2}\left \| \frac{\tau \sum_{p=1}^{m}\left ( u _{H}^{p}-F_{H}u _{h}^{p} \right )+\tau \sum_{p=1}^{m+1}\left ( u _{H}^{p}-F_{H}u _{h}^{p} \right )}{2}  \right \|_{\text{DG}}^{2}\\
    &\leq \left \|u _{H}^{0}-F_{H}u _{h}^{0}  \right \|_{s}^{2}\\
    &\quad +2\left ( \left ( m+1 \right )\left \| u _{H}^{1}-F_{H}u _{h}^{1}-\left ( u _{H}^{0}-F_{H}u _{h}^{0} \right ) \right \|_{s} +\tau ^{2}\sum_{p=1}^{m}G^{p}\right )\underset{0\leq p \leq N}{\rm{max}}\left \| \pi\left ( u _{H}^{p}-F_{H}u _{h}^{p} \right ) \right \|_{s}\\
    &\quad +2\left ( \tau ^{2} \sum_{p=1}^{m}D^{p}\right )\underset{0\leq p\leq N_1-1}{\rm{max}}\left \| P\left (\frac{ u _{H}^{p+1}-F_{H}u _{h}^{p+1}+u _{H}^{p}-F_{H}u _{h}^{p}}{2} \right ) \right \|_{s}.
    \end{aligned}
\end{equation}
By the Lemma \ref{lemmaa}, we note that
\begin{equation}
    N_1\left \| u _{H}^{1}-F_{H}u _{h}^{1}+\left ( u _{H}^{0}-F_{H}u _{h}^{0} \right )  \right \|_{s}\leq c_5\left (\tau+4\tau ^{2}\left(1- \frac{c_{1}\tau ^{2}}{4c_{3}H^2} \right)^{-1}\left(2c_5+1\right)+\frac{H^{2}}{c_2}c_5\right ).
\end{equation}
Again using $G^{p}\leq \frac{c_1}{2}\left(\frac{H^{2}}{c_2}+1\right)$ in (\ref{equ:gm}) and $D^{p}\leq \frac{\sqrt{2}H}{\sqrt{c_2}}+c_5 $ in (\ref{equ:dm}), we possess
\begin{equation}
    \tau ^{2}\sum_{p=1}^{N-1}\left ( G^{p} +D^{p} \right )\leq c_5\left(\tau+4\tau ^{2}\left(1- \frac{c_{1}\tau ^{2}}{4c_{3}H^2} \right)^{-1}\left(2c_5+1\right)+\frac{H^{2}}{c_2}c_5 \right ).
\end{equation}
Combining Theorem \ref{thm1}, we have
\begin{equation}
\begin{aligned}
&\quad \underset{0\leq p\leq N_1-1}{\rm{max}}\left \| u _{H}^{p}-F_{H}u _{h}^{p}\right \|_{s}\\
&\leq c_5\left (\left \| u _{H}^{0}-F_{H}u _{h}^{0} \right \|_{s} +\tau+4\tau ^{2}\left(1- \frac{c_{1}\tau ^{2}}{4c_{3}H^2} \right)^{-1}\left(2c_5+1\right)+\frac{H^{2}}{c_2}c_5\right ).
\end{aligned}
\end{equation}
Using $u _{H}^{0}-F_{H}u _{h}^{0}=S_{H}u _{h}^{0}-\left ( u _{h}^{0} -u _{H}^{0}\right ),$ we have
\begin{equation}
   \begin{aligned}
    &\quad  \underset{0\leq p\leq N_1-1}{\rm{max}}\left \|u _{H}^{p}-F_{H}u _{h}^{p}\right \|_{s}\\
    & \leq 12c_5\left ( \left \| u _{h}^{0}-u _{H}^{0} \right \|_{s}+ \underset{0\leq p\leq N_1}{\max}\left \| S_{H} u _{h}^{p}\right \|_{s}+\tau+4\tau ^{2}\left(1- \frac{c_{1}\tau ^{2}}{4c_{3}H^2} \right)^{-1}\left(2c_5+1\right)+\frac{H^{2}}{c_2}c_5\right ).
   \end{aligned}
\end{equation}
Using $u _{H}^{p}-F_{H}u _{h}^{p}=S_{H}u _{h}^{p}-\left ( u _{h}^{p} -u _{H}^{p}\right ),$ we possess
\begin{equation}
    \begin{aligned}
    &\quad \underset{0\leq p\leq N_1-1}{\rm{max}}\left \|u _{H}^{p}-u _{h}^{p}\right \|_{s}\\
    &\leq \underset{0\leq p\leq N_1-1}{\rm{max}}\left \| \frac{ u _{H}^{p+1}-F_{H}u _{h}^{p+1}+u _{H}^{p}-F_{H}u _{h}^{p}}{2}\right \|_{s}+\underset{0\leq p\leq N}{\rm{max}}\left \| S_{H}u _{h}^{m} \right \|_{s}\\
    &\leq 12c_5\left ( \left \| u _{H}^{0}-u _{h}^{0} \right \|_{s}+\underset{0\leq p\leq N_1}{\rm{max}}\left \| S_{H}u _{h}^{m} \right \|_{s}+\tau+4\tau ^{2}\left(1- \frac{c_{1}\tau ^{2}}{4c_{3}H^2} \right)^{-1}\left(2c_5+1\right)+\frac{H^{2}}{c_2}c_5 \right ).\\
    \end{aligned}
\end{equation}
Since $L^{2}:u _{h}^{0}\rightarrow u _{H}^{0},$ for any $u _{h}^{0}\in V_{\rm{cem}},$ we have
\begin{equation}
    \left \| u _{H}^{0}-u _{h}^{0} \right \|_{s}\leq \left \| S_{H}u _{h}^{0} \right \|_{s}\leq 12c_5 \left(\tau+4\tau ^{2}\left(1- \frac{c_{1}\tau ^{2}}{4c_{3}H^2} \right)^{-1}+\frac{H^{2}}{c_2}c_5\right).
\end{equation}
Hence we finish the proof.
\end{proof}

\section{Numerical results}\label{sec:Numerical results}
In this section, we use two numerical examples to show the accuracy of our multiscale model reduction approach. The situation in this paper is anisotropic, and for the sake of clarity of comparison, we consider two complicated layered anisotropic models shown in \autoref{fig:model4} and \autoref{fig:model21} respectively. We note that the parameters in \autoref{fig:model4} are $C_{11}, C_{13}, C_{33}, C_{55}$, the corresponding values in \autoref{fig:model21}  are  $C_{11}, 0.5C_{13}, 0.5C_{33}, 0.25C_{55}$.
The computational domain  $\Omega=\left(0,6000\right)^2,\  T = 0.4$ and $ \ \rho =1.$ We divide this spatial domain into $30\times30$ coarse elements   and each coarse element is divided into $20\times20$ fine elements,  therefore the fine-grid size is $h=10$.

The following relative error in $L^{2}$ norm and energy norm $H^{1}$ are used to quantify  the accuracy of our method: 
\begin{equation*}
\begin{aligned}
     error_{L^{2}}&=\frac{\left\|u_{H}-u_{h}\right\|_{L^{2}(\Omega)}}{\left\|u_{h}\right\|_{L^{2}(\Omega)}},\\
     error_{H^{1}}&=\frac{\left\|u_{H}-u_{h}\right\|_{H^{1}(\Omega)}}{\left\|u_{h}\right\|_{H^{1}(\Omega)}}.
\end{aligned}
\end{equation*}

 \autoref{tablesymbol} lists  the notations of some parameters.

\begin{table}[H]
\setlength{\abovecaptionskip}{0pt}
\setlength{\belowcaptionskip}{10pt}
\centering
\caption{\label{tablesymbol}Simplified description of symbols}
\begin{tabular}{l|llllll}
\cline{1-2}
Parameters                                        & Symbols &  &  &  &  &  \\ \cline{1-2}
Number of oversampling layers                    & $Nol$      &  &  &  &  &  \\
Number of basis functions in each coarse element & $Nbf $    &  &  &  &  &  \\
Length of each coarse element size               & $L $      &  &  &  &  &  \\
 \cline{1-2}
\end{tabular}
\end{table}

To fully examine the influence of the number of oversampling layers($Nol$), $Nol$ is taken here as $m=5, m=6, m=7$ and $m=8.$ In all these cases, we use $20$ auxiliary basis functions in each coarse block to construct  corresponding local multiscale basis functions. For fair comparison, we use  same time step for both the CEM-GMsDGM and the fine-scale methods, as we only consider spatial ascending scales in this paper. However, we note that multiscale basis functions can be used for different source terms and boundary conditions, which will result in significant computational savings. The source function is
\begin{equation}
f\left(t, x, y\right)=e^{-10^{2}\left(\left(x-0.5\right)^{2}+\right(y-0.5\left)^{2}\right)}\left(1-2 \pi^{2} f_{0}^{2}\left(t-2 / f_{0}\right)^{2}\right) e^{-\pi^{2} f_{0}^{2}\left(t-2 / f_{o}\right)^{2}},
\end{equation}
where the center frequency is chosen to be $f_{0}=10$.  The time step size $\Delta  t=10^{-4}$ and the penalty paramter $\gamma = 2.$

\subsection{Model 1}

\begin{figure}[H]
\centering
\includegraphics[width=0.9\textwidth]{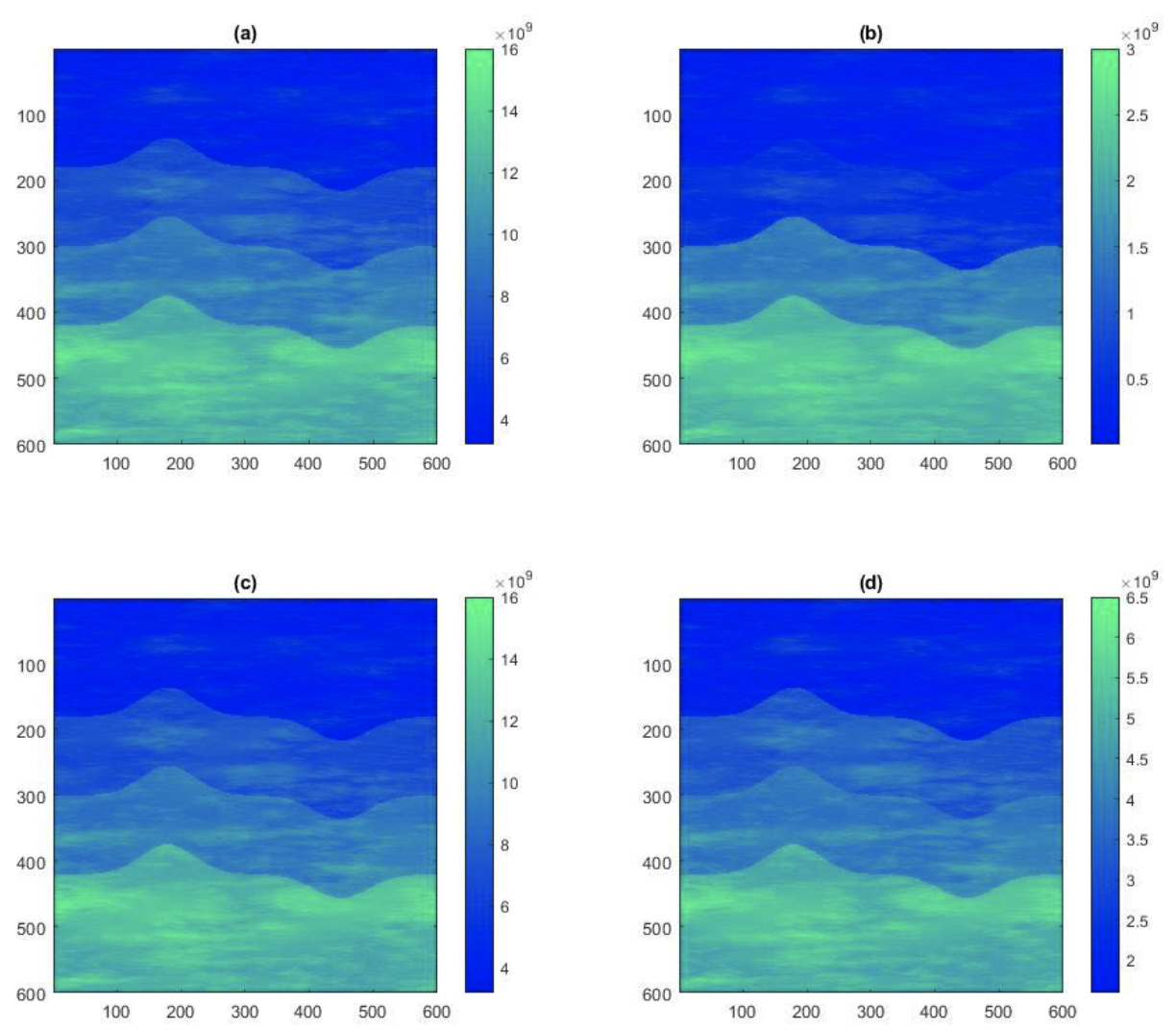} 
\caption{Model 1: An anisotropic elastic model. Panels (a)-(d) represent $C_{11}, C_{13}, C_{33}, $ and $ C_{55},$  respectively, the units here is GPa.}\label{fig:model4}
\end{figure}

\autoref{fig:wave4}  and \autoref{fig:wave22} depict the first and second component of the two numerical solutions (the fine scale solution and the multiscale solution) for Model 1 at the final time $T = 1.25$.  
\autoref{NolL}   investigates the $H^{1}$ error and $L^{2}$ error of our new method with the two variables layers, the number of oversampling layers and coarse element size. 

To study the convergence of our proposed multiscale method, we compare the coarse-scale approximation with the fine-grid solution.  In \autoref{NolL}, we set the number of basis functions on each coarse grid to $6$. As the number of oversampling layers increases from $7$ to $8$  and the coarse element size doubles from $1/15$ to $1/30$, the $L^ 2$ error decreases from  $58.3498\%$ to $31.2438\%$ and the $H^ 1$ error decreases from $49.2020\%$ to $13.4207\%.$  In the process of increasing the number of oversampling layers by $6$ at a time to $8$ from an equal number of $1$  and decreasing coarse element size  from $1/30$ at a time, the $L^2$ error is reduced by a factor of $5$ each time, from $31.2438\%$ to $1.9321\%.$  
The difference is that the $H^1$ error is reduced by a factor of $25$  from $13.4207\%$  to $0.6111\%$  as the number of oversampling layers is $6$ and the coarse element size is $1/30$ to $8$ and the coarse element size is $1/120.$   It can been observed that the method results in good accuracy and desired convergence in error.  \autoref{fig:wave4} and \autoref{fig:wave22} depict the numerical solutions by the fine-scale formulation and the coarse-scale formulation at the final time $T = 1.25.$  The comparison suggests that our new method provides very good accuracy at a reduced computational expense.

\begin{table}[H]
\setlength{\abovecaptionskip}{0pt}
\setlength{\belowcaptionskip}{10pt}
\centering
\caption{\label{NolL} $Nbf=6.$}
\begin{tabular}{cllllllclllll}
\cline{1-10}
$Nol$ &  &  & $L$ &  &  & \multicolumn{1}{c}{$error_{L^{2}}$} &  &  & $error_{H^{1}}$ &  &  &    \\ \cline{1-10}
5   &  &  & $1/15$ &  &  &   58.3498\%   &  &  &  49.2020\%  &  &  &  \\
6  &  &  &  $1/30$ &  &  &  31.2438\%   &  &  &  13.4207\% &  &  &  \\
7  &  &  &  $1/60$ &  &  &   8.9200\%  &  &  &  1.8301\%   &  &  &  \\ 
8  &  &  &  $1/120$  &  &  & 1.9321\%   &  &  &  0.6111\%      &  &  & \\
\cline{1-10}
\end{tabular}
\end{table}

\begin{figure}[H] 
\centering 
\includegraphics[width=1\textwidth]{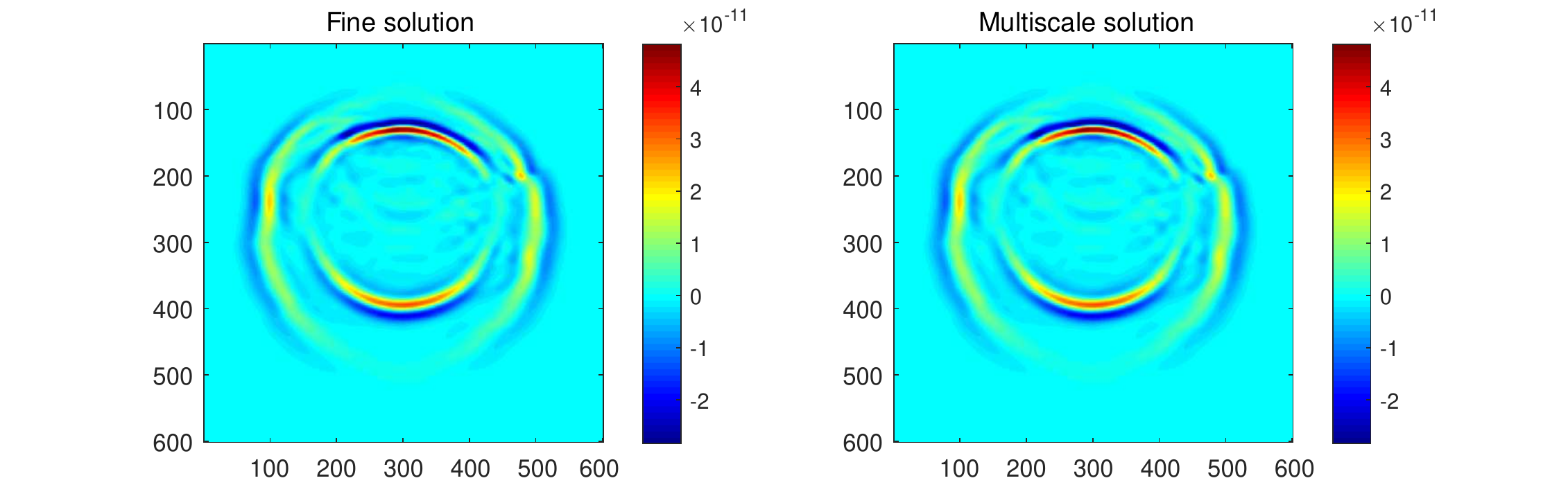} 
\caption{Comparison of first component of the reference solution and the multiscale solution: the left figure is the first component of the reference solution at $T=1.25$, the right figure is the first component of the CEM-GMsFEM solution at $T=1.25$. }
\label{fig:wave4} 
\end{figure}

\begin{figure}[H] 
\centering 
\includegraphics[width=1\textwidth]{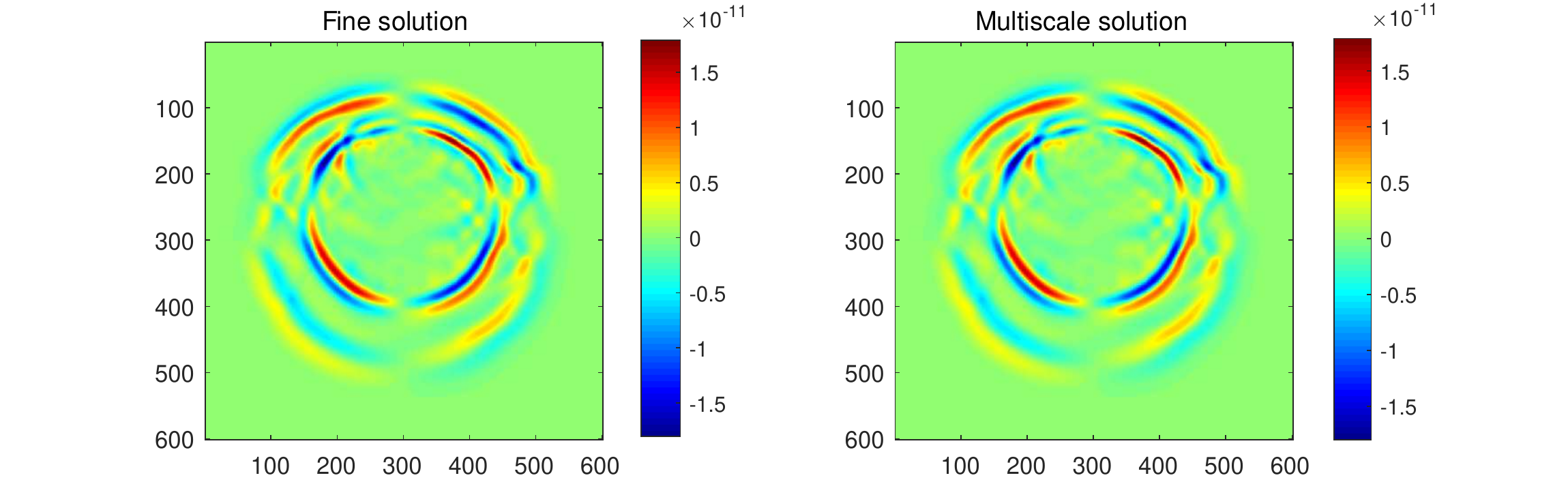} 
\caption{Comparison of second component of the reference solution and the multiscale solution: the left figure is the second component of the reference solution at $T=1.25$, the right figure is the second component of the CEM-GMsFEM solution at $T=1.25$. }
\label{fig:wave22} 
\end{figure}

\subsection{Model 2}

\autoref{fig:model2ff1}  and \autoref{fig:model2ff2} depict the first and second component of the two numerical solutions (the fine scale solution and the multiscale solution) for Model $2$ at the final time $T = 1.25$.  
\autoref{NolL2}   investigates the $H^{1}$ error and $L^{2}$ error of our new method with the two variables layers, the number of oversampling layers and coarse element size.

The solid experimental data for model 2 demonstrate our method's computational cost reductions as well as the correctness of the numerical findings.  When we first set the coarse scale size to $1/15$ and the number of oversampled grids to $5$, the $ L^2$ error  is $72.4120\%$  and  $H^1$ is  $59.3120\%$,  both of which are more than $10\%$  higher than in model  $1$  with the same parameter settings. However, as the coarse grid size decreases and the number of oversampling layers increases, the $L^2$ error and  $H^1$ error both drop dramatically, especially when the coarse grid size is reduced to  $1/120$  and the number of oversampling layers  is increased to $8$,  the $L^2$ error has been reduced by a factor of approximately  $35$ and the $H^1$  error has been reduced by a factor of $70$.

\begin{figure}[H]
\centering
\includegraphics[width=1.0\textwidth]{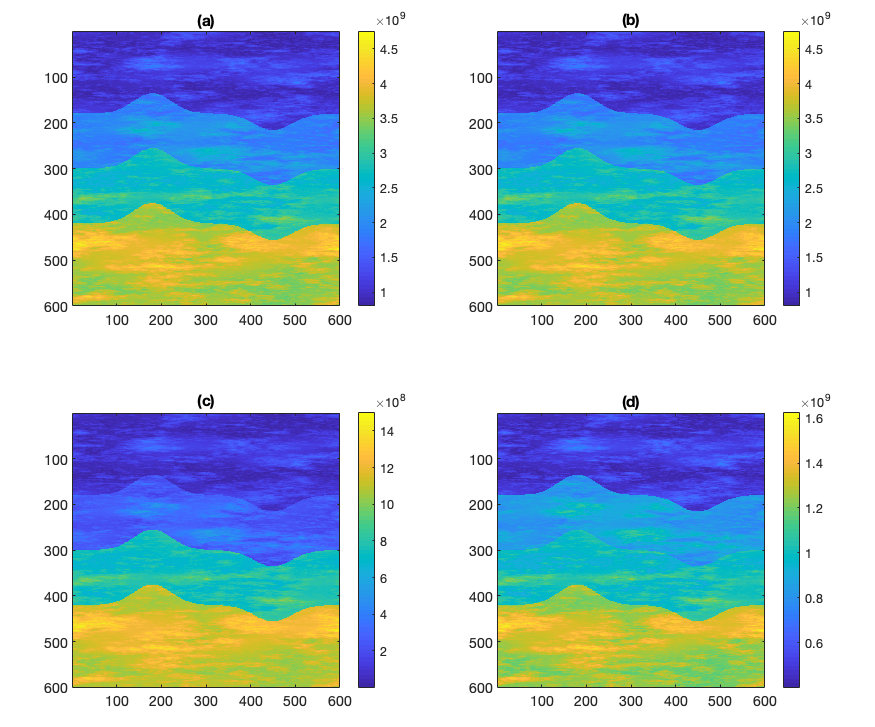} 
\caption{Model2: An anisotropic elastic model. Panels (a)-(d) represent $C_{11}, 0.5C_{13}, 0.5C_{33}, $ and $ 0.25C_{55},$  respectively, the units here is GPa.}\label{fig:model21}
\end{figure}


\begin{table}[H]
\setlength{\abovecaptionskip}{0pt}
\setlength{\belowcaptionskip}{10pt}
\centering
\caption{\label{NolL2} $Nbf=6.$}
\begin{tabular}{cllllllclllll}
\cline{1-10}
$Nol$ &  &  & $L$ &  &  & \multicolumn{1}{c}{$error_{L^{2}}$} &  &  & $error_{H^{1}}$ &  &  &    \\ \cline{1-10}
5   &  &  & $1/15$ &  &  &72.4120\%   &  &  & 59.3120\%  &  &  &  \\
6  &  &  &  $1/30$ &  &  &  35.4210\%   &  &  &  15.2218\% &  &  &  \\
7  &  &  &  $1/60$ &  &  &   8.9341\%  &  &  &  2.3101\%   &  &  &  \\ 
8  &  &  &  $1/120$  &  &  & 2.0912\%   &  &  &  0.8246\%      &  &  & \\
\cline{1-10}
\end{tabular}
\end{table}

\begin{figure}[H] 
\centering 
\includegraphics[width=1.0\textwidth]{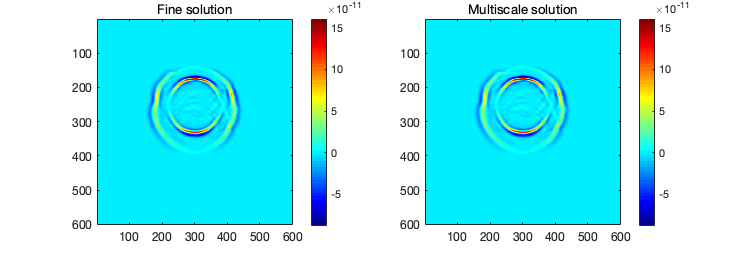} 
\caption{Comparison of first component of the reference solution and the multiscale solution: the left figure is the first component of the reference solution at $T=1.25$, the right figure is the first component of the CEM-GMsFEM solution at $T=1.25$. }
\label{fig:model2ff1} 
\end{figure}

\begin{figure}[H] 
\centering 
\includegraphics[width=1.0\textwidth]{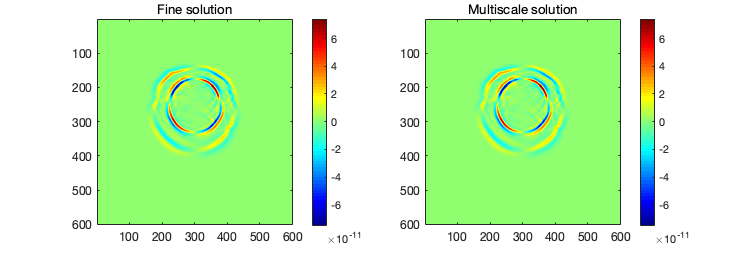} 
\caption{Comparison of second component of the reference solution and the multiscale solution: the left figure is the second component of the reference solution at $T=1.25$, the right figure is the second component of the CEM-GMsFEM solution at $T=1.25$. }
\label{fig:model2ff2} 
\end{figure}

The conclusions obtained are as follows:
\begin{itemize}
    \item The number of oversampled layers and the length of the coarse scale have an inverse relationship with error; more specifically, as the number of oversampled layers grows and the length of the coarse scale drops, the accuracy improves.

    \item In our testing, when the number of basis functions is $ 6$, the number of oversampling layers is pretty high and ranges within $20\%$, and raising the number of oversampling layers to $8$ yields good error outcomes. As a consequence, this experiment illustrates that our technique produces extremely accurate findings while saving a large amount of processing time.
\end{itemize}


\section{Conclusion}\label{sec:Conclusion}
In this paper, we present a multiscale approach to elastic wave propagation applicable to anisotropy, also known as the CEM-GMsDGM, which is a general multiscale model reduction method for the heterogeneous wave equation in the DG framework. We explore the solution of multiscale experimental basis functions for discontinuities on a coarse grid, a process that can be solved by constrained energy minimisation in the oversampling region to obtain multiscale experimental basis functions.The CEM-GMsDGM is shown to be stable and spectrally convergent in both theory and practice. Moreover, to verify the effectiveness of the method, we devise a modified Möbius model, and it turns out that the accuracy of the multiscale solution is closely related to the number of oversampling layers used in the modelling. The level of accuracy can be controlled by varying this number, which is important in applications where more approximate results can be accepted.
\section*{Acknowledgments}

The research of Eric Chung is partially supported by the Hong Kong RGC General Research Fund (Project numbers 14304719 and 14302620) and CUHK Faculty of Science Direct Grant 2021-22.  

\small

\bibliography{secondpaper}

\end{document}